\documentclass[a4paper, reqno]{amsart}

\usepackage{latexsym}
\usepackage{amsmath}
\usepackage{amsfonts}
\usepackage{amssymb}
\usepackage{amsxtra}
\usepackage{mathrsfs}
\usepackage{eucal}
\usepackage{tikz} \usetikzlibrary{calc} \tikzset{>=latex} \usetikzlibrary{backgrounds}
\usepackage[all]{xy}
\usepackage{color}
\usepackage{braket}
\usepackage{graphics, setspace}
\usepackage{etoolbox, textcomp}
\usepackage{comment}
\usepackage{changepage}
\usepackage{xcolor}
\definecolor{rouge}{rgb}{0.85,0.1,.4}
\definecolor{bleu}{rgb}{0.1,0.2,0.9}
\definecolor{violet}{rgb}{0.7,0,0.8}
\usepackage[colorlinks=true,linkcolor=bleu,urlcolor=violet,citecolor=rouge]{hyperref}
\usepackage{cleveref}
\usepackage{lscape}
\usepackage{subcaption}
\usepackage{hyperref}

\RequirePackage[numbers,sort&compress]{natbib}

\definecolor{rouge}{rgb}{0.85,0.1,.4}

\newtheorem{theorem}{Theorem}[section]

\newtheorem{definition}[theorem]{Definition}

\newtheorem{proposition}[theorem]{Proposition}
\newtheorem{corollary}[theorem]{Corollary}

\newtheorem{assumption}[theorem]{Assumption}
\newtheorem{remark}[theorem]{Remark}

\newcommand{\DiscMod}[1]{\mathcal{D}_{#1}}
\newcommand{\IrrMod}[1]{\mathcal{L}_{#1}}
\newcommand{\StagMod}[1]{\mathcal{S}_{#1}}
\newcommand{\TypMod}[1]{\mathcal{E}_{#1}}

\newcommand{\brac}[1]{\left( #1 \right)}

\newcommand{\abs}[1]{\left| #1 \right|}

\newcommand{\Gr}[1]{\Bigl[ #1 \Bigr]} 

\newcommand{\fuse}{\mathbin{\times_V}}

\newdimen{\Virwidth}
\newcommand{\vfuscoeff}[2]{
  \settowidth{\Virwidth}{$\mathrm{Vir}$}
  \mathsf{N}_{#1}^{\mathrm{Vir} \hspace{-\Virwidth} \hphantom{#1} #2}
}
\newcommand{\vpfuscoeff}[3]{
  \settowidth{\Virwidth}{$v$}
  \mathsf{N}_{#1}^{#3 \hspace{-\Virwidth} \hphantom{#1} #2}
}

\def\CC{\mathbb{C}}

\def\HH{\mathbb{H}}

\def\RR{\mathbb{R}}

\def\ZZ{\mathbb{Z}}

\newcommand\cB{\mathcal{B}}
\newcommand\cC{\mathcal{C}}
\newcommand\cD{\mathcal{D}}
\newcommand\cE{\mathcal{E}}
\newcommand\cF{\mathcal{F}}
\newcommand\cG{\mathcal{G}}

\newcommand\cM{\mathcal{M}}

\newcommand{\hopflink}{{\,\text{\textmarried}}}

\newcommand{\ch}{\textup{ch}}
\newcommand{\Vir}{\textup{Vir}}

\newcommand\Irr{\textup{Irr}}

\newcommand{\tr}{\textup{tr}}

\newcommand\typ{\textup{typ}}
\newcommand{\Vect}{\textup{Vec}}

\newcommand\Hom{\textup{Hom}}

\newcommand\quash[1]{}

\renewcommand\a\alpha
\renewcommand\b\beta
\newcommand\ga\gamma
\renewcommand\d\delta
\newcommand\D\Delta

\newcommand{\sfaut}{\sigma}                   
\newcommand{\sfmod}[2]{\sfaut^{#1}(#2)}       

\newcommand{\lra}{\longrightarrow}

\newcommand{\dses}[3]{0 \lra #1 \lra #2 \lra #3 \lra 0} 

\newcommand{\sltwo}{\mathfrak{sl}_2}

\begin{document}

\newcommand{\ssl}{\mathfrak{sl}_{2|1}}
\newcommand{\vir}{\mathcal{M}}
\newcommand{\virp}{\vir_{>0}}
\newcommand{\virb}{\vir_{\ge0}}
\newcommand{\virz}{\vir_0}
\newcommand{\virn}{\vir_{<0}}

\newcommand{\uea}[1]{\mathcal{U}(#1)}

\newcommand{\hw}{highest-weight}
\newcommand{\foh}[1]{\textcolor{blue}{#1}}  
\newcommand{\jy}[1]{\textcolor{red}{#1}}
\newcommand{\oc}{\mathcal{O}_c}
\newcommand{\ocfin}{\oc^{\textup{fin}}}
\newcommand{\ocleft}{\oc^{\textup{L}}}
\newcommand{\ocright}{\oc^{\textup{R}}}
\newcommand{\cofcat}{\mathcal{C}_1}

\title[  ]{Resolving Verlinde's formula of logarithmic CFT}
\author{Thomas Creutzig}
\address{Department Mathematik, Friedrich-Alexander-Universit\"at
Erlangen, Cauerstr. 11
91058 Erlangen}
\email{thomas.creutzig@fau.de}

\date{}
\maketitle

\begin{abstract}
Verlinde's formula for rational vertex operator algebras  computes the fusion rules from the modular transformations of characters. 
In the non semisimple and non finite case, a logarithmic Verlinde formula has been proposed together with David Ridout. In this formula one replaces simple modules by their resolutions by standard modules. 
Here and under certain natural assumptions this conjecture is proven in generality. 

The result is illustrated in the examples of the singlet algebras and of the affine vertex algebra of $\sltwo$ at any admissible level, i.e. in particular the Verlinde conjectures of \cite{CR1, CreFal13} are true. In the latter case it is also explained how to compute the actual fusion rules from knowledge of the Grothendieck ring. 
\end{abstract}

\section{History of Verlinde's formula for VOAs}

Two-dimensional conformal field theories (CFTs) rose to importance in the 1980's since the world-sheet quantum field theory of a string is such a conformal field theory. CFT also quickly established itself as a rich source for new and exciting mathematical structures with  monstrous moonshine as a famous example \cite{Bor}.  Vertex operator algebras (VOAs) are a rigorous notion of the chiral or symmetry algebra of a CFT. 
A most influential pair of works have been those of Eric Verlinde \cite{Ver88} and Moore-Seiberg \cite{MS88}.  At that time one was interested in rational CFTs, that is theories with only a finite number of simple modules for the underlying VOA $V$ and such that every module is completely reducible. The quantities of interest in physics are often the correlation functions or conformal blocks, certain meromorphic functions on Riemann surfaces. In particular a basic constraint of the CFT being well-defined is that torus one-point functions (characters) span a vector-valued modular form. Verlinde's celebrated observation is that these modular transformations seem to govern the fusion rules of modules. For this let $V=M_0, M_i, \dots, M_n$ be the list of inequivalent simple $V$-modules and let 
\[
\ch[M_i](\tau, v) = \tr_{M_i}(o(v) q^{L_0-\frac{c}{24}}), \qquad q = e^{2\pi i \tau}, v \in V
\]
be the character of $M_i$. Here $o(v)$ is the zero-mode associated to $v$, $\tau$ is in the upper half $\HH$ of the complex plane $\CC$, $L_0$ is the Virasoro zero-mode and $c$ is its central charge. Modularity means in particular that
\[
\ch[M_i]\left(-\frac{1}{\tau}, v\right) = \tau^h \sum_{j= 0}^n S_{i, j} \ch[M_j](\tau, v)
\]
with $h$ the conformal weight of $v$. Verlinde conjectured that this modular $S$-matrix governs the fusion rules
\[
M_i \otimes_V M_j \cong \bigoplus_{k=0}^n N_{i, j}^{\ \ k} M_k, \qquad \qquad N_{i, j}^{\ \ k} = \sum_{\ell = 0}^n \frac{S_{i, \ell}, S_{j, \ell} S_{\ell, k}^*}{S_{0, \ell}}.
\]
 $*$ denotes complex conjugation. This defines Verlinde's algebra of characters, 
 \[
 \ch[M_i] \times_V \ch[M_j] := \sum_{k = 0}^n N_{i, j}^{\ \ k} \ch[M_k]
 \]
 so that Verlinde's formula is rephrazed as the Verlinde algebra of characters is isomorphic to the fusion ring, i.e.
\begin{equation}\label{eq:verchrat}
\ch[M_i] \times_V \ch[M_j] = \ch[M_i \otimes_V M_j]
\end{equation}
One can reformulate this into the algebra of quantum dimensions
\begin{equation}\label{eq:verqdimrat}
\frac{S_{i, \ell}}{S_{0, \ell}} \frac{S_{j, \ell}}{S_{0, \ell}}  = \sum_{k=0}^n N_{i, j}^{\ \ k}\  \frac{S_{k, \ell}}{S_{0, \ell}}.
\end{equation}
 Moore and Seiberg were able to formalize the axiomatics of CFT, which led to the notion of modular tensor categories \cite{Tur}. A modular tensor category $\cC$ is a semi-simple, finite, braided ribbon category for which the categorical $S$-matrix, defined by the Hopf link
\tikzset{ar/.style={<-}}
\tikzset{br/.style={->}}
\begin{equation}\label{hopflink}\nonumber
\begin{tikzpicture}

\draw[ar] (0, 3) arc (40:390:.7cm);
\node[text width=4.2cm] at (-1.8, 2.2)    {$S^\hopflink_{ij}:=\qquad M_i$};

\draw[br] (0, 2) arc (220:570:.7cm);
\node[text width=3cm] at (2.9, 2.3)    {$M_j\quad\in\CC$};

\end{tikzpicture}
\end{equation}
is non-degenerate. Verlinde's formula with $S$ replaced by $S^\hopflink$ is nothing but an exercise. 

This development provided a nice route map for researchers in the theory of VOAs. The aim was now to establish that sufficiently nice VOAs, which are nowadays called strongly rational \cite{CreLog16},  have modular tensor categories as representation categories, their characters form vector-valued modular forms and Verlinde's conjecture is indeed true. Modularity of modified characters was settled by Yongchang Zhu \cite{ZhuMod96}, a theory of tensor categories has been developed by Huang-Lepowsky \cite{HL2, HL3, H1} for rational VOAs and then been generalized by Huang-Lepowsky-Zhang \cite{HLZ1, HLZ2, HLZ3, HLZ4, HLZ5, HLZ6, HLZ7, HLZ8} and rigidity and Verlinde's formula have been proven by Yi-Zhi Huang \cite{H2, H3}. These results are surely the most influential ones of the theory of rational VOAs.

However, almost all VOAs are not of the rational type. This was already realized by physicists a long time ago and in particular if the Virasoro zero-mode does not act semi-simply, then logarithmic singularities might appear in correlation functions \cite{GurLog93}, hence the name logarithmic CFT. Quickly people were interested in an analogue of Verlinde's formula for logarithmic theories \cite{Koh88,Fuchs:2003yu, Gainutdinov:2007tc} and conjectural explanations \cite{Gainutdinov:2016qhz, CreLog16}. A prime example seemed to be WZW theories at fractional levels, as they have modules whose characters "are" meromorphic Jacobi forms \cite{KacMod88}. 
The VOA of a fractional level WZW theory is an affine VOA at admissible level and it had been mainly studied in the case of $\sltwo$. 
However a naive application of Verlinde's formula gave negative fusion coefficients \cite{Koh88}. Such theories were then unfortunately discarded as physically sick and it took more than two decades to successfully revisit the problem \cite{CreMod12,CR1}. Characters are a priori only formal power  series and not meromorphic Jacobi forms and they coincide with the expansion of certain meromorphic Jacobi forms in an appropriate domain. Moreover the modules that correspond in this way to meromorphic Jacobi forms  are only a very small subset of all modules of the VOA. The idea of \cite{CreMod12,CR1}, motivated from \cite{CreBra07, CreRel11},  is to consider modular transformations on an uncountable set of standard modules and then to get a modular $S$-kernel for all simple modules by considering resolutions by standard modules. This allowed  to conjecture a Verlinde formula for the fractional theories of $\sltwo$. Since then this standard formalism to Verlinde's formula has been applied to examples of many different VOAs  \cite{ CreFal13, CreReg14, CanFusI15, CanFusII15, Auger:2019gts,  CM2, Fas, Alfes:2012pa, 
Fehil, Kawasetsu:2021qls, RidOSP17, RidBos14, RidMod13} (and in some variants as the modular properties are of different flavours, see \cite{CreLog13, RidVer14} for early reviews). 

With these observations there was now the task to develop a sensible theory of representations for VOAs that are associated to logarithmic CFT with the proof of the logarithmic Verlinde's formula as an ultimate goal. 
Understanding the representation theory of a given category $\cC$ of modules of a VOA $V$ amounts to firstly classifying its simple, injective and projective objects, then to establish existence of vertex tensor category structure and finally study this structure, e.g. establish its rigidity, compute fusion rules and maybe more. All of these tasks are quite involved and they only have been completely achieved in the examples of the triplet VOAs \cite{AdaTri08, TW}, the singlet VOAs \cite{Ad-singlet, CMY4, CMY6}, the $\beta\gamma$-ghosts \cite{Allen:2020kkt}, the $\cB_p$-algebras \cite{CreCos13,CMY3}, the affine VOA of $\mathfrak{gl}_{1|1}$ \cite{CreRel11, CMY1} and the category of weight modules of affine $\sltwo$ at admissible levels \cite{ACK23, C23, CMY24, Flor24}.

In all the studies of non semi-simple categories of VOA modules the existence of a good realization of the VOA $V$ had always been a great aid. This means that $V$ embeds conformally into another VOA $A$, whose representation category is completely known. For example the triplet VOAs embed into lattice VOAs. In particular every $A$-module is a $V$-module, but much more structure can usually be inferred from this. It had been instrumental for proving  Kazhdan-Lusztig correspondences, e.g. between the representation categories of the singlet VOAs and unrolled small quantum groups of $\sltwo$ \cite{CLR}, and it allows, under certain conditions, to infer rigidity of $V$-modules from rigidity of $A$-modules \cite{CMSY24}. This work adds a logarithmic Verlinde formula to this list of powerful implications of good realizations.  
Namely, in the next section, the main Theorem, that is Theorem \ref{thm:Ver},  will be stated and proven. It asserts that under certain assumptions a sensible notion of quantum dimension exists and they obey an algebra of quantum dimension in complete analogy to the rational setting \eqref{eq:verqdimrat}. 
The standard Verlinde formula is a corollary, that is Corollary \ref{cor:Ver}, it uses the quantum dimension to define a Verlinde kernel so that integrating characters against this kernel gives the Verlinde algebra of characters in analogy to \eqref{eq:verchrat} with the crucial difference that the sum is replaced by an integral.

The final section illustrates the findings in the important examples of the singlet algebras and the affine VOA $L_k(\sltwo)$ of $\sltwo$ at admissible levels $k$. In particular the Verlinde conjectures of \cite{CR1, CreFal13} are true. In the case of $L_k(\sltwo)$ at admissible level we explain how Verlinde's formula together with rigidity implies the actual fusion rules. We demonstrate this for  simple and projective modules. 

{\bf Acknowledgements} I am very grateful  to David Ridout and Antun Milas for past collaborations on this topic which formed the foundation of my understanding. I am equally very grateful to Shashank Kanade, Jinwei Yang and Robert McRae for all the collaborations on vertex tensor categories, which have been invaluable. 
Finally I very much appreciate the useful comments of Yi-Zhi Huang and Drazen Adamovic. 

The fusion rules of weight modules for $L_k(\mathfrak{sl}_2)$ at admissible levels are computed with VOA techniques in \cite{Flor24}. 
We agreed to submit our manuscripts to the arXiv simultaneously.

\section{Standard resolution of Verlinde's algebra of quantum dimensions}

Vertex algebras often have uncountable infinitely many modules and so any Verlinde formula will involve an integral over the space of objects, i.e. we need to introduce some technicalities to say what we mean by a Verlinde formula. The Grothendieck ring of a tensor category  $\cC$ with exact tensor product is denoted by $K(\cC)$. 
A vertex tensor category is in particular an abelian braided tensor category (with tensor product $\otimes  = \otimes_{P(1)}$). For simplicity we work with vertex tensor categories, but the argument for vertex tensor supercategories is the same.  

A very useful tool in studying vertex algebras $V$ with non semi-simple representation categories are good realizations, that is conformal embeddings $V \hookrightarrow A$ into simpler vertex algebras $A$; simpler in the sense that their representation category of interest is semisimple. 
One usually assumes that $A$ is haploid: $\Hom_\cC(V, A) = \CC$. In this setting and under the assumption that there is a vertex tensor category $\cC$ of $V$-modules in the sense of \cite{HLZ1, HLZ2, HLZ3, HLZ4, HLZ5, HLZ6, HLZ7, HLZ8} that contains $A$ as an object one can identify $A$ with a commutative algebra in $\cC$ \cite{HKL} and moreover the subcategory of $\cC_A^{\text{loc}} \subset \cC_A$ is a vertex tensor category of modules of the VOA $A$ \cite{CKM}. There are then two functors
\[
\cF: \cC \rightarrow \cC_A, \qquad \cG: \cC_A \rightarrow \cC,
\]
the induction and restriction functors. The induction functor is monoidal 
\[
\cF(X) \otimes_A \cF(Y) \cong  \cF(X \otimes_V Y)
\]
for any two objects $X, Y$ in $\cC$. Here we denote the tensor product bi functor in $\cC$ by $\otimes_V$ and the one in $\cC_A$ by $\otimes_A$. The composition $\cG \circ \cF$ is just tensoring with $A$,  $\cG(\cF(X)) = A \otimes_V X$.

\begin{definition}\label{def:Verss} 
Let $A$ be a vertex operator algebra and $\cD$ a vertex tensor category of $A$-modules.
Assume
\begin{enumerate}
\item $\cD$ is semisimple and rigid.
\item $T := \text{Irr}(\cD)$
\item  There exists a function $S: T \times T \rightarrow \CC, (X, Y) \mapsto S_{X, Y}$, the $S$-kernel, s.t. $S_{A, Y} \neq 0$ for any $Y \in T$.
\item The quantum dimension of the object $X$ is defined to be $q^A_X: T \rightarrow \CC, Y \mapsto \frac{S_{X, Y}}{S_{A, Y}}$. Let $Q_A$ be the linear span of the $q^A_X$, $Q_A = \text{span}_\CC\{q^A_X | X \in \Irr(\cC)\}$. Assume that the $q^A_X$  are linearly independent.
%
\end{enumerate}
We say that the category $\cD$ admits a {\bf semisimple Verlinde algebra of quantum dimensions} if 
$Q_A$ is closed under multiplication, $q^A_X q^A_Y \in Q_A$ for any $X, Y  \in \Irr(\cD)$, and 
the map
\[
K(\cD) \rightarrow Q_A, \qquad X \mapsto q^A_X
\]
is a ring homomorphism. 
%
\end{definition}
We will give examples of such semisimple Verlinde algebras of quantum dimensions in  section \ref{sec:examples}. We want to lift this semisimple Verlinde algebra to non-semisimple settings. For this we will use resolutions.

\subsection{Resolutions and quantum dimensions}

We first introduce some vocabulary for resolutions and quantum dimensions.
\begin{definition}
Let $V \hookrightarrow A$ be a conformal embedding of vertex operator algebras such that $A$ is an object in a vertex tensor category $\cC$ of $V$-modules. 
\begin{enumerate}
\item 
A module $X$ in $\cC$ is said to be standard (with respect to $A$) if there exists an object $Y$ in $\cC_A$ with $X \cong  \cG(Y)$. We  write $X^A$ for any object with the property that $X \cong  \cG(X^A)$
\item A standard module $X$ in $\cC$ is said to be basic if $X \cong  \cG(Y)$ for a simple  object $Y$ in $\cC_A$. The set of basic standard modules will be denoted by $T$. 
\item Let $M^{\text{typ}}$ be the linear $\mathbb Z$-span of isomorphism classes of basic standard modules. Basic standard modules are said to be linearly independent if the map  $M^{\text{typ}} \rightarrow K(\cC)$, $X \mapsto [X]$ is injective. 
\item 
A composition series $0 = X_0 \subset X_1 \subset \dots \subset X_n = X$ of an object $X$ in $\cC$ is called standard if each composition factor $X_i/X_{i-1}$ is basic standard.
\item A resolution
\[
 \cdots \xrightarrow{f_4} X_3 \xrightarrow{f_3} X_2 \xrightarrow{f_2} X_1 \xrightarrow{f_1} X_0 \xrightarrow{f_0} X \rightarrow 0
\]
s. t. each $X_i$ is standard, is called standard.  The associated standard chain complex is $X_\bullet = \cdots \xrightarrow{f_4} X_3 \xrightarrow{f_3} X_2 \xrightarrow{f_2} X_1 \xrightarrow{f_1} X_0$.
If $X$ is already standard then one can take  $X_\bullet = \cdots 0 \rightarrow 0 \rightarrow 0 \rightarrow X$. 
\item A standard resolution is said to have finite multiplicity if for all  basic standard modules $Y$
\[
m_{X_\bullet}(Y) :=\sum_{i=0}^\infty  [X_i: Y] < \infty.
\]
\item 
The index of a basic standard module $Y$  in a standard finite multiplicity resolution $X_\bullet$ is  
\[
I_{X_\bullet}(Y) :=\sum_{i=0}^\infty  (-1)^i[X_i: Y] \in \mathbb Z.
\]
\end{enumerate}
\end{definition}
The contragredient dual of a $V$-module $X$ is denoted by $X'$. If $\cC$ is rigid, then $V \cong V'$ and the contragredient dual and dual of a module coincide. 
\begin{assumption}\label{assumption}
Let $V \hookrightarrow A$ be a conformal embedding of vertex operator algebras such that $A$ is $\mathbb Z$-graded by conformal weight,  haploid and an object in a vertex tensor category $\cC$ of $V$-modules. Set $\cD = \cC_A^\text{loc}$.
    Assume
\begin{enumerate}
\item $\cC$ is locally finite, has enough projectives  and is rigid.
\item $\cD$ is semisimple and admits a semisimple Verlinde algebra of quantum dimensions.
\item $\Irr(\cC_A) = \Irr(\cD)$.
\item There is a one-to-one correspondence $\tau: \Irr(\cC) \rightarrow \Irr(\cD)$, such that the top of $\cG(\tau(X))$ is $X$. And there exists a simple standard module $X$, such that both $X$ and $X'$ are projective. 
\item Basic standard modules are linearly independent.
\item Assume that $\cD$ is graded by some abelian group $G$, that is 
 \[
 \mathcal D \cong \bigoplus_{g \in G} \mathcal D_g, \qquad X \otimes_A Y \in \cD_{gh} \ \text{for} \ X \in \cD_g, Y \in \cD_h
 \]
 and assume that there is a group homomorphism $\rho: G \rightarrow (\mathbb R, +)$.
\end{enumerate}
\end{assumption}
By Theorem 3.14 of \cite{CMSY24} these assumptions imply that $\cC_A$ is rigid. 
Let $X, Y$ be basic standard objects, that is there exists objects $X^A, Y^A$ in $\cD$ with $\cG(X^A) = X, \cG(Y^A)=Y$. Set $S_{X, Y}:=S_{X^A, Y^A}$ and similarly $q_X^A := q_{X^A}^A$. Extend this definition linearly to direct sums and then to objects $X$ in $\cC_A$ by choosing a standard composition series of $X$ and setting $q_X^A$ to be the sum of the quantum dimensions of standard composition factors. We will later show that all standard composition series of $X$ are equivalent, Prop. \ref{prop:typcomp}.
 \begin{definition} \label{def:resolvedqdim} Under the Assumption.
 \begin{enumerate}
 \item Let $Y$ be a finite length object in $\cC_A$. Then by point (3) of the assumption it  has a finite composition series with composition factors basic standard modules. Let $\textup{supp}(Y):= \{ g \in G| \ \textup{there exists} \ X \in T \cap \cD_g \ \textup{with} \ [Y:X] \neq 0\} \subset G$ and let $d(Y):= \textup{min}\{ \rho(g) | g \in \textup{supp}(Y)\}$.
 Define $d$ on standard modules via $d(\cG(Y)) := d(Y)$.
 \item $Y$ is called homogeneous if $|\textup{supp}(Y)| =1$.
  Let $z$ be a formal variable, define $q_Y^A(z) := z^{\rho(Y)} q_Y^A$ for homogeneous objects $Y$ and extend linearly. 
  \item Consider a standard resolution of $Y$
\[
Y_\bullet =  \qquad \qquad \cdots Y_3 \rightarrow Y_2 \rightarrow Y_1 \rightarrow Y_0 \rightarrow Y \rightarrow 0.
\]
 such a resolution is called $\rho$-ordered if $m_{Y_\bullet}(Z) \neq 0$ implies that $\rho(Z) = d(Y_0) \mod 1$ and the sequence $(d(Y_i))_{i \in \mathbb Z_{\geq 0}}$ is monotonously increasing and never stabilizes. It is called strictly $\rho$-ordered if $(d(Y_i))_{i \in \mathbb Z_{\geq 0}}$ is strictly monotonously increasing.
 \item A $\rho$-ordered standard resolution is said to
admit a quantum dimension if 
$$
q^A_{Y_\bullet}(t, z) := z^{-d(Y_0)}\sum_{i=0}^\infty (-t)^i q^A_{Y_i}(z) \ \in Q_A[z][[t]]$$ 
and if  $q^A_{Y_\bullet}(t, z)$ converges point-wise for $|t|, |z|<1$ to a rational function $f_{Y_\bullet}(t, z)\in Q_A(t, z)$, such that there is a countable set $E \subset T$, such that the limits $\lim\limits_{z \rightarrow 1^-}  \lim\limits_{t \rightarrow 1^-} f_{Y_\bullet}(M)(t, z)$ and $\lim\limits_{t \rightarrow 1^-}  \lim\limits_{tz\rightarrow 1^-} f_{Y_\bullet}(M)(t, z)$ exist and coincide, $\lim\limits_{z \rightarrow 1^-}  \lim\limits_{t \rightarrow 1^-} f_{Y_\bullet}(M)(t) = \lim\limits_{t \rightarrow 1^-}  \lim\limits_{z\rightarrow 1^-} f_{Y_\bullet}(M)(t, z)$, for all $M$ in $T\setminus E$.
The limit is denoted by $q_{Y_\bullet}^A$ and called the 
quantum dimension of the resolution $Y_\bullet$.
\end{enumerate}
 \end{definition}
\begin{remark}\label{rem:finitemult}
    Since a $\rho$-ordered standard resolution never stabilizes it must be of finite multiplicity and hence $q^A_{Y_\bullet}(t, z)$ is in $Q_A[z][[t]]$ as indicated. 
\end{remark} 
Let $X_\bullet \rightarrow X \rightarrow 0$ be a standard resolution of $X$ and let $X_\bullet^{\text{triv}}= \qquad \dots 0 \rightarrow 0 \rightarrow 0 \rightarrow X$, so that $X_\bullet^{\text{triv}} \rightarrow X \rightarrow 0$ is the trivial resolution of $X$ (it is not standard unless $X$ is standard). Clearly there is a quasi-isomorphism $X_\bullet \rightarrow X_\bullet^{\text{triv}}$ and hence if basic standard modules are linearly independent then the Index $I_{X_\bullet}(Y)$ does not depend on a choice of finite multiplicity resolution. 
Similarly we need that the quantum dimension, if it exists, not to depend on a choice of $\rho$-ordered standard resolution. 
\begin{proposition}\label{prop:limit}  Retain assumption \ref{assumption}. Let $Y_\bullet, \tilde Y_\bullet$ be two $\rho$-ordered standard resolutions of $Y$, then $q_{Y_\bullet}^A = q_{\tilde Y_\bullet}^A$.
\end{proposition}
\begin{proof}
Consider  $R(t, z) :=q^A_{Y_\bullet}(t, z) -z^{d(\tilde Y_0) - d(Y_0)}q^A_{\tilde Y_\bullet}(t, z)$, since both $Y_\bullet$ and $\tilde Y_\bullet$ are resolutions of the same object $Y$ the difference has to be of the form 
\[
R(t, z) = (1-t) \sum_{i=0}^\infty t^i q_{i}(z)
\]
for certain polynomials $q_i(z) \in Q_A[z]$.
Since $q^A_{Y_\bullet}(t, z), q^A_{\tilde Y_\bullet}(t, z)$ both converge to rational functions in $Q_A(t, z)$ for $|t|, |z| <1$, the same is true for $R(t, z)$ and hence for $\sum_{i=0}^\infty t^i q_{i}(z)$. We have to make sure that $\sum_{i=0}^\infty t^i q_{i}(z)$ in reduced form  can't have a factor of the form $\frac{1}{1-t} = \sum_{i=0}^\infty t^n$.

Let $n_i = \text{min}\{ d(Y_i), d(\tilde Y_i)\}$, then $q_{i}(z) \in z^{n_i}Q_A[z]$. Since both sequences $(d(Y_i))_{i \in \mathbb Z_{\geq 0}}, (d(\tilde Y_i))_{i \in \mathbb Z_{\geq 0}}$ are monotonously increasing and never stabilizing the same is true for $(n_i)_{i \in \mathbb Z_{\geq 0}}$. It follows that $R(t,z)$ is of the form
\[
R(t, z) = (1-t) \sum_{i=0}^\infty z^i p_i(t) \ \in  Q_A[t][[z]]
\]
for certain polynomials $p_i(t)$ in $Q_A[t]$. Hence $\lim_{t \rightarrow 1^-}\sum_{i=0}^\infty t^i q_{i}(z)= \sum_{i=0}^\infty z^i p_i(1) \in Q_A[[z]]$ and so $\sum_{i=0}^\infty t^i q_{i}(z)$ in reduced form cannot  have a factor of the form $\frac{1}{1-t} = \sum_{i=0}^\infty t^n$.
In particular $\lim_{t \rightarrow 1^-}R(t, z) =0$ for all $z$ with $|z| <1$. Hence $\lim_{z \rightarrow 1^-}\lim_{t \rightarrow 1^-}R(t, z) = \lim_{t \rightarrow 1^-}\lim_{z \rightarrow 1^-}R(t, z) =0$.
\end{proof}

Set $q^A_Y := q^A_{Y_\bullet}$, $q_Y:= \frac{q^A_Y}{q^A_V}$ and $S_{Y, M}:= q^A_Y(M)S_{A, M}$, so that $q_Y(M) = \frac{S_{Y, M}}{S_{V, M}}$ . Let $Q = \textup{span}_{\mathbb C}\{ q_Y | Y \ \in \ \text{Obj}(\cC)\}$.
   Extend the definitions linearly, that is $q_{X \oplus Y} = q_X + q_Y$ 
and $S_{X \oplus Y, M} = S_{X, M}  + S_{Y, M}$.
   Note that these definitions do not depend on the choice $Y_\bullet$ by Proposition \ref{prop:limit}.

\begin{theorem}\label{thm:Ver} 
Retain assumption \ref{assumption}. In particular let $V$ be a vertex operator algbera and $\cC$ a vertex tensor category of $V$-modules with $\cD := \cC_A^\text{loc}$. \
Assume that $V$ has a standard strict $\rho$-ordered resolution $V_\bullet \rightarrow V \rightarrow 0$,
$V_\bullet = \qquad\qquad \cdots  V_3 \rightarrow V_2 \rightarrow V_1 \rightarrow V_0$, admitting a quantum dimension.
Then 
\begin{enumerate}
   \item Every indecomposable object $Y$ in $\cC$ has a $\rho$-ordered standard resolution $Y_\bullet$ that admits a quantum dimension. The total complex of $X_\bullet \otimes_V Y_\bullet$  for any two $\rho$-ordered resolutions $X_\bullet, Y_\bullet$ is also $\rho$-ordered.
   \item The quantum dimension factors through the Grothendieck ring of $\cC$, that is for any short exact sequence $0 \rightarrow X \rightarrow Y \rightarrow Z \rightarrow 0$ in $\cC$ one has $q^A_Y = q_X^A  + q_Z^A$.
\item
The category $\cC$ admits a {\bf Verlinde algebra of quantum dimensions}, that is
$Q$ is closed under multiplication, $q_X q_Y \in Q$ for any $X, Y  \in \Irr(\cC)$, and 
the map
\[
K(\cC) \rightarrow Q, \qquad X \mapsto q_X
\]
is a ring homomorphism. 
\end{enumerate}
\end{theorem}
\begin{proof} The rest of this section is devoted to the proof of this main Theorem.

\begin{proposition}\textup{\cite[Lemma 2.8]{CLR}\label{prop:CLR}}
Let $V \hookrightarrow A$ be a conformal embedding of vertex operator algebras such that $A$ is an object in a vertex tensor category $\cC$ of $V$-modules. 
Let $X$ be an object in $\cC$ and $Y$ one in $\cC_A$. Then there exists a multiplication $m_{X \otimes Y}: A \otimes_V (X \otimes_V \cG(Y)) \rightarrow X \otimes_V \cG(Y))$, such that $(X \otimes_V \cG(Y), m_{X \otimes Y})$ is an object in $\cC_A$ and as such  $(X \otimes_V \cG(Y), m_{X \otimes Y}) \cong \cF(X) \otimes_A Y$.
\end{proposition}
Applying $\cG$ to $(X \otimes_V \cG(Y), m_{X \otimes Y}) \cong \cF(X) \otimes_A Y$ shows that standard modules form a tensor ideal:
\begin{corollary}\label{cor:typprod}
Let $V \hookrightarrow A$ be a conformal embedding of vertex operator algebras such that $A$ is an object in a vertex tensor category $\cC$ of $V$-modules. 
Let $X, Y$ be objects in $\cC$ and $Y$ be standard, then so is $X \otimes_V Y$. 
\end{corollary}



\begin{proposition}\label{prop:typcomp}
$\cG(X)$ for $X$ in $\textup{Obj}(\cC_A)$ has up to equivalence a unique standard composition series. 
\end{proposition}
\begin{proof}
Let $X$ be an object in $\cC_A$. Since $\cC$ is of finite length, $\cC_A$ has to be of finite length as well.  The assumption $\Irr(\cC_A) = \Irr(\cD)$ implies that $X$ has a finite composition series with local composition factors. The restriction functor is the identity on modules and morphisms and simply forgets the algebra action on $A$. It is in particular an exact functor and so a Jordan-H\"older series of $X$ maps under the restriction functor to a composition series of $\cG(X)$ where all composition factors are standard modules. 
Let $Y_1$ be a simple direct summand of the top of $W_0:=\cG(X)$ and set $Z_1:= \cG(\tau(Y_1))$. By assumption $Y_1$ is the top of $Z_1$ and $Z_1$ is the only simple object of $\cD$ with this property. Thus the only possible standard composition factor having $Y_1$ as top is $Z_1$.  Let $Y_2$ be a simple direct summand of  the top of $W_1 := \ker(\cG(X) \rightarrow Z_1)$ and proceed inductively. This provides a standard composition series. Since in each step $\text{top}(W_i)/Y_{i+1}$ is contained in the top of $W_{i+1}$ and since each standard composition factor is uniquely determined by its top this procedure is unique up to ordering.
\end{proof}
The restriction functor maps the set $\Irr(\cD)$ to the set of basic standard modules. By assumption \ref{assumption} point (4) this map needs to be injective and by definition of basic standard this map is surjective, hence
\begin{corollary}\label{cor:JH}
Let $X, Y \in \textup{Obj}(\cC_A)$, s.t. $\cG(X) \cong \cG(Y)$, then $X, Y$ have equivalent Jordan-H\"older series. 
\end{corollary}
Let $K(\cC_{\text{typ}})$ be the $\ZZ$-span of basic standard modules and define a map from standard modules to $K(\cC_{\text{typ}})$, mapping $X$ to $[X]_\typ$ where $[X]_\typ$ is the sum of basic standard composition factors of $X$. By Proposition \ref{prop:typcomp} this is well-defined and by Corollary \ref{cor:typprod} the tensor product in $\cC$ induces a ring structure on it.

Let $K(\cD)$ be the Grothendieck ring of $\cD$. For an object $N$ in $\cD$ define a new product $\cdot_N : K(\cD) \times K(\cD) \rightarrow K(\cD), ([X], [Y]) \mapsto [N \otimes_A X \otimes_A Y]$. Denote $K(\cD, N)$ the ring $K(\cD)$ with product $\cdot_N$. 

For the remainder of the proof we set $N := \cF(A)$.
\begin{proposition}
 The map
\[
\varphi: K(\cC_{\typ}) \rightarrow K(\cD, N), \qquad [\cG(X)]_\typ \mapsto [X]
\]
is a ring isomorphism.
\end{proposition}
\begin{proof}
Recall that Theorem 3.14 of \cite{CMSY24} applies thanks to assumption \ref{assumption} and thus $\cC_A$ is rigid. 

 Let $X, Y$ be objects in $\cC_A$, such that $\cG(X) \cong \cG(Y)$, then by Corollary \ref{cor:JH} $X$ and $Y$ have equivalent Jordan-H\"older series and in particular their images in 
$K(\cD)$ coincide $[X] \cong [Y]$. This map is a ring homomorphism from the split Grothendieck ring of $\cC_A$ to $K(\cD)$, since $\cC_A$ is rigid.
Since basic standard modules are in one-to-one correspondence with simple objects in $\cD$, the map $\varphi$ is bijective. We need to show that it is a homomorphism of rings. 

We apply Proposition \ref{prop:CLR} several times. First note that the case $X = A$ says 
 that $\cG(\cF(\cG(Z))) \cong A \otimes_V \cG(Z) \cong \cG(\cF(A) \otimes_A Z) = \cG(N \otimes_A Z)$ for any object $Z$ in $\cC_A$. Hence
 $ [\cF(\cG(Z))] \cong [N \otimes_A Z]$.
 Consider two standard modules $\cG(X), \cG(Y)$  as well as $\cG(X) \otimes_V \cG(Y) \cong \cG(\cF(\cG(X)) \otimes_A Y)$. 
Hence
\begin{equation} \nonumber
\begin{split}
[\cG(X) \otimes_V \cG(Y)]_\typ &\cong [\cG((\cG(X) \otimes_V \cG(Y), m_{\cG(X) \otimes Y}))]_\typ 
\cong [\cG(\cF(\cG(X)) \otimes_A Y)]_\typ \\
&\cong  \varphi^{-1}\left([\cF(\cG(X)) \otimes_A Y]  \right) 
\cong  \varphi^{-1}\left([ N \otimes_A  X \otimes_A Y]  \right) 
\cong  [\cG\left( N \otimes_A  X \otimes_A Y  \right)]_\typ.
\end{split}
\end{equation}

\vspace{-5mm}
\end{proof} 
\begin{proposition}
    $\rho(g)$ in $\mathbb Z$ for $g \in \textup{supp}(N)$.
\end{proposition}
\begin{proof}
    Since $\cC$ is of finite length and $A$ is an object in $\cC$ also $A \otimes_V A$ is an object in $\cC$ and hence of finite length. Especially $N$ must be of finite length in $\cC_A$ and thus $\text{supp}(N)$ is finite. Assume that the statement of the Proposition is false. We will see that this is impossible. That is assume there exists $g \in \text{supp}(N)$
    with $\rho(g) \notin \mathbb Z$ and choose $g$ minimal with that property, that is $h \in \text{supp}(N)$ with $\rho(h) \notin \mathbb Z$ implies that $\rho(h) \geq \rho(g)$. 

    Apply the induction functor (it is exact since $\otimes_V$ is exact) to the strictly $\rho$-ordered resolution $V_\bullet$ of $V$ to get a resolution of $A$ in $\cC_A$. 
\begin{equation}\nonumber
\cdots \cF(V_3) \rightarrow \cF(V_2) \rightarrow \cF(V_1) \rightarrow \cF(V_0) \rightarrow A \rightarrow 0.
\end{equation}
Let $V_i^A$ be such that $\cG(V_i^A) = V_i$.
Applying the restriction functor  and recalling that $\cG(\cF(V_i)) \cong \cG(N \otimes_A V_i^A)$ gives  a standard resolution $A_\bullet$ of $A$
\begin{equation}\label{eq:Abullet}
A_\bullet = \qquad\qquad \cdots \cG(N \otimes_A V^A_3) \rightarrow  \cG(N \otimes_A V^A_2) \rightarrow  \cG(N \otimes_A V^A_1) \rightarrow  \cG(N \otimes_A V^A_0) \rightarrow  A \rightarrow 0.
\end{equation}
Since $d(N \otimes_A V^A_i) = d(N) + d(V_i)$ and $d(V_{i+1}) > d(V_i)$ also $d(N \otimes_A V^A_{i+1})>d(N \otimes_A V^A_i)$. 
In particular $A_\bullet$ must be of finite multiplicity. Since it is a resolution of $A$ the index satisfies $I_{A_\bullet}(X) = 0$ for $X \in T\setminus \{A\}$.

Since $\cC_A$ is rigid $X \otimes_A Y$ is non-zero for any two non-zero objects $X, Y$ in $\cC_A$.
It follows that $N \otimes_A V^A_i$ has a standard (non-zero) composition factor, call it $Z_i$, with $\rho(Z) = d(V_i) + \rho(g) \in \rho(g) + \mathbb Z$ and by the minimality assumption on $g$ any other composition factor $Y_i$ of $N \otimes_A V_i$ with $d(Y_i) \in \rho(g) + \mathbb Z$ has to satisfy $d(Y_i) \geq \rho(g)$. In particular $Z_0$ cannot be a composition factor of 
$N \otimes_A V^A_i$ for $i>0$. Thus $I_{A_\bullet}(Z_0) \neq 0,$ a contradiction.
\end{proof}
\begin{corollary}\label{cor:supp} ${}$
    \begin{enumerate}
        \item For $X, Y$ in $T$, $\textup{supp}(X \otimes_V Y) \subset d(X) + d(Y) + \mathbb Z$.
        \item For $X$ in $\textup{Irr}(\cC)$ and $Y$ in $T$, $\textup{supp}(X \otimes_V Y) \subset d(\tau(X)) + d(Y) + \mathbb Z$.
        \item For $X$ in $\textup{Irr}(\cC)$ and $Y$ in $T$, $\textup{supp}(P_X \otimes_V Y) \subset d(\tau(X)) + d(Y) + \mathbb Z$, where $P_X$ is the projective cover of $X$.
        \item For $X$ in $\textup{Irr}(\cC)$ and $Y$ in $T$, $\textup{supp}(M \otimes_V Y) \subset d(\tau(X)) + d(Y) + \mathbb Z$, where $M$ is in the same block as $X$.
        \end{enumerate}
    \end{corollary}
\begin{proof}
    The first statement follows from the previous proposition together with $\cG(X) \otimes_V \cG(Y) \cong \cG(N \otimes_A X \otimes_A Y)$.

    The second statement follows from the first one and exactness of $\otimes_V$, that is $\cG(\tau(X)) \otimes_VY \rightarrow X \otimes_V Y \rightarrow 0$ is exact. 

    For the third statement choose an $Z \in T$ that is projective as an object in $\cC$
    and $Z'$ is as well. Then $Z \otimes_V  Z' \otimes_V X$ is projective and surjects onto $P_X$. The claim follows since $\text{supp}(Z \otimes_V \otimes Z' \otimes_V X) \subset d(Z) + d(Z') +d(\tau(X)) + \mathbb Z = d(\tau(X)) + \mathbb Z$.

    For the last statement consider $0 \rightarrow Z \rightarrow M \rightarrow X \rightarrow 0$ non-split and $X, Z$ simple in $\cC$. Since $\text{Hom}_\cC(P_X, \bullet)$ is exact there must be a non-zero morphism from $P_X$ to $M$ and since $Z$ has to embed into the image one has $\text{supp}(Z \otimes_VY) \subset d(\tau(X)) + d(Y) + \mathbb Z$ for any $Y$ in $T$. By exactness also $\text{supp}(M \otimes_VY) \subset d(\tau(X)) + d(Y) + \mathbb Z$ for any $Y$ in $T$. The claim follows.
    \end{proof}

We prove the first statement of our Theorem.
\begin{proposition}
    Every indecomposable object $Y$ in $\cC$ has a $\rho$-ordered standard resolution $Y_\bullet$ that admits a quantum dimension. The total complex of $X_\bullet \otimes_V Y_\bullet$  for any two $\rho$-ordered resolutions $X_\bullet, Y_\bullet$ is also $\rho$-ordered.
\end{proposition}
\begin{proof}
    Let $Y$ be an indecomposable object in $\cC$. We then have the resolution $Y_\bullet = Y \otimes_V V_\bullet$ of $Y$. Since standard modules form a tensor ideal, Corollary \ref{cor:typprod}, this is a standard resolution.  The terms of the resolution are $Y_i = Y \otimes_V V_i$ and since $Y \otimes_V V_i \cong \cG(\cF(Y) \otimes_A V_i^A)$ one has $d(Y_i) = d(\cF(Y)) + d(V_i)$ and hence $d(Y_{i+1})> d(Y_i)$. The support of each $Y_i$ lies in $d(Y) + \mathbb Z$ by Corollary \ref{cor:supp}. Since $\cG(\cF(Y_i)) \cong \cG(N \otimes_A Y^A_i)$ and since $q_{N \otimes_A Y^A_i}^A = q_N^A q_{Y_i}^A$ it follows that $z^{Y_0}q^A_{Y_\bullet}(t, z) = q_N(z)q^{V_0}q_{V_\bullet}(r, z)$ and so $Y_\bullet$ admits a quantum dimension.  

    The last statement is clear since $X_i \otimes_V Y_j \cong \cG(N \otimes_A X_i^A \otimes_A  Y_j^A)$ implies $d(X_i \otimes_V Y_j) = d(N) + d(X_i) + d(Y_j)$.
\end{proof}
The second statement of the Theorem follows.
\begin{corollary}
    The quantum dimension factors through the Grothendieck ring of $\cC$, that is for any short exact sequence $0 \rightarrow X \rightarrow Y \rightarrow Z \rightarrow 0$ in $\cC$ one has $q^A_Y = q_X^A  + q_Z^A$.
\end{corollary}
\begin{proof}
Consider a non-split short exact sequence $0 \rightarrow X \rightarrow Y \rightarrow Z \rightarrow  0$. Tensoring this with  $V_\bullet$ of $V$ gives a short exact sequence of 
standard resolutions $0 \rightarrow X_\bullet \rightarrow Y_\bullet \rightarrow Z_\bullet \rightarrow  0$, where $X_\bullet = X \otimes_V V_\bullet, Y_\bullet = Y \otimes_V V_\bullet, Z_\bullet = Z \otimes_V V_\bullet$. 
Hence $q_{Y_\bullet}(t, z) = q_{X_\bullet}(t, z)+q_{Z_\bullet}(t, z)$ and in particular $q_Y = q_Z + q_Z$.
\end{proof}


Via the ring homomorphism $\varphi$, 
$q_N^A  q_X^A q_Y^A = q^A_{X \otimes_V Y}$
for standard indecomposable modules $X, Y$. In general for two $\rho$-ordered standard resolutions $X_\bullet, Y_\bullet$ that admit quantum dimensions 
\begin{equation}\label{eq:qdimhomres}
q_N^A(z)  q_{X_\bullet}^A(t, z) q_{Y_\bullet}^A(t, z) = q^A_{\text{Tot}(X_\bullet \otimes_V Y_\bullet)}(t, z)
\end{equation}
and in particular $\text{Tot}(X_\bullet \otimes_V Y_\bullet)$ admits a quantum dimension as well.
The resolution \eqref{eq:Abullet} satisfies all the assumptions of Proposition \ref{prop:limit} and  so  
\[
1 = q_A^A = \lim_{t \rightarrow 1^-} \lim_{z \rightarrow 1^-} q_N^A q_{V_\bullet}^A = q_N^Aq_V^A.
\]
This and taking the limits $t, z \rightarrow 1^-$ of \eqref{eq:qdimhomres} gives
\begin{corollary}${}$
\begin{enumerate}
    \item $q_N^Aq_V^A = 1$, in particular $q_V^A(M) \neq 0$ for all $M \in T\setminus E$.
    \item For  $X, Y$ in $\text{Obj}(\cC)$,
    $q_N^A q_X^A q_Y^A = q^A_{X \otimes_V Y}$.
\end{enumerate}
\end{corollary}
Set $q_X = \frac{q_X^A}{q_V^A}$ then the algebra of quantum dimension follows
\begin{equation}\nonumber
    \begin{split}
 q_X q_Y &= \frac{q_X^A}{q_V^A} \frac{q_Y^A}{q_V^A} 
 = \frac{q^A_N q^A_X q^A_Y}{q_V^A} 
 = \frac{q^A_{X \otimes_V Y}}{q_V^A} 
 = q_{X \otimes_V Y}.
    \end{split}
\end{equation}
\vspace{-10mm}

\end{proof}

\vspace{2mm}

\section{Standard Verlinde's formula of characters}

We now want to replace the algebra of quantum dimensions by some analytic expression. For this we extend Definition \ref{def:Verss}.
Let $\ch[X]= \tr_X(q^{L_0-\frac{c}{24}})$ be the character of an object $X$ in $\cC$ (and if these characters are not linearly independent then also graded in addition by some Jacobi variable or refined by an insertion of a zero-mode $o(v)$). 

\begin{definition}\label{def:Verss2} 
Let $A$ be a vertex operator algebra and $\cD$ a vertex tensor category of $A$-modules.
Assume
\begin{enumerate}
\item $\cD$ is semisimple and rigid.
\item The set of isomorphism classes of simple objects, $T := \text{Irr}(\cD)$, is a  measure space with measure $\mu$.
\item  There exists a function $S: T \times T \rightarrow \CC, (X, Y) \mapsto S_{X, Y}$, the $S$-kernel, s.t. $S_{A, Y} \neq 0$ for any $Y \in T$.
Extend the notion to direct sums as before, that is $S_{W \oplus X, Y} = S_{W, Y} + S_{X, Y}$.
\item $\ch[X]$ converges to a function on $D^1 = \{ q \in \mathbb C | |q| <1 \}$ for all $X$ in $T$, so that 
$$\ch: T \times D^1: \mathbb C, \qquad (X, q) \mapsto \ch[X](q), $$ becomes a function on $T \times D^1$.
\item Let $V^\ch = \text{span}_{\mathbb C}\{ \ch[X] | X \in \text{Obj}(\cD) \}$.
A function $f: T \times T \rightarrow \mathbb C$ is called a convolution kernel if 
\[
\int_T \mu(Y) \left(\int_T \mu(Z) f(Y, Z) \ch[Z]\right)
\]
exists and is in $V^\ch$. Assume that the functions $I_XY, Z) = S^{*}_{X, Y}  S_{Y, Z}$, $I^*_X(Y, Z) = S_{X, Y}  S^{*}_{Y, Z}$ and $N_{U, X}(Z, W) =  \frac{S_{U, Z}  S_{X, Z} S^{*}_{W, Z}}{S_{A, Z}}$ are convolution kernels for all  $U, X \in \text{Obj}(\cD)$, that is 
 for all $U, X \in \text{Obj}(\cD)$ and $q \in D^1$, the integrals 
 \begin{equation} \nonumber
\begin{split}
&\int_{T} \mu(Y) \left(\int_{T} \mu(Z)\   S^{*}_{X, Y}  S_{Y, Z} \ch[Z]  \right), \qquad \int_{T} \mu(Y) \ \left(\int_{T} \mu(Z)   S_{X, Y}  S^{*}_{Y, Z} \ch[Z]  \right),  \\
&\int_{T} \mu(Z) \left(\int_{T} \mu(W)\   \frac{S_{U, Z}  S_{X, Z} S^{*}_{W, Z}}{S_{A, Z}} \ch[W]  \right) 
\end{split}
\end{equation}
exist and are in $V^\ch$.
\end{enumerate}
Then
\begin{enumerate}
\item[(a)] 
The $S$-kernel is called {\bf unitary} if for all $X \in \text{Obj}(\cD)$
\begin{equation}\label{eq:unitary}
\begin{split}
\int_{T} \mu(Y) \left(\int_{T} \mu(Z)\   S^{*}_{X, Y}  S_{Y, Z} \ch[Z]  \right)  &=  \ch[{X}] =
\int_{T} \mu(Y) \ \left(\int_{T} \mu(Z)   S_{X, Y}  S^{*}_{Y, Z} \ch[Z]  \right)  
\end{split}
\end{equation}
\item[(b)] The {\bf Verlinde algebra of characters} is defined to be for $X, Y$ in $X \in \text{Obj}(\cD)$
 \[
\ch[X] \times_A \ch[Y] := \int_{T} \mu(Z) \left(\int_{T} \mu(W)\   \frac{S_{X, Z}  S_{Y, Z} S^{*}_{Z, W}}{S_{A, Z}} \ch[W]  \right)  .
\]
\end{enumerate}
\end{definition}

We note that the set $E$ in Definiton \ref{def:resolvedqdim} can now  be replaced by a set of measure zero. 
If $\cD$ admits a semisimple Verlinde algebra of quantum dimensions, that is
\[
\frac{S_{X, Z}}{S_{A, Z}} \frac{S_{Y, Z}}{S_{A, Z}} = \frac{S_{X \otimes_A Y, Z}}{S_{A, Z}} = \sum_{W \in T} F_{X, Y}^{\ \ W}\  \frac{S_{W, Z}}{S_{A, Z}}
\]
with $F_{X, Y}^{\ \ W}$ the fusion rules in $\cD$, 
\[
X \otimes_A Y \cong \bigoplus_{W \in T} F_{X, Y}^{\ \ W} \ W. 
\]
Then the Verlinde algebra is identified with the fusion algebra as follows:
\begin{equation}
    \begin{split}
      \ch[X] \times_A \ch[Y] &= \int_{T} \mu(Z) \left(\int_{T} \mu(W)\   \frac{S_{X, Z}  S_{Y, Z} S^{*}_{Z,W}}{S_{A, Z}} \ch[W]  \right) \\
      &= \int_{T} \mu(Z) \left(\int_{T} \mu(W)\  \sum_{U \in T} F_{X, Y}^{\ \ U}\  {S_{U, Z}}{S^{*}_{Z, W}} \ch[W]  \right)\\
      &= \int_{T} \mu(Z) \left(\int_{T} \mu(W)\   {S_{X \otimes_A Y, Z}}{S^{*}_{Z, W}} \ch[W]  \right)\\
      &= \ch[X \otimes_A Y].
    \end{split}
\end{equation}


Let now $X$ be a standard module, then by Proposition \ref{prop:typcomp} it has up to equivalence a unique standard composition series
and so $\ch[X]$ is the sum of the characters of its composition factors. 
The {\bf Verlinde algebra of characters} is defined to be for $X$ a standard module and $Y$ any object in $\cC$.
\begin{equation}\label{eq:verlinde}
\begin{split}
\ch[X] \times_V \ch[Y] := \int_{T} \mu(Z) \left(\int_{T} \mu(W)\   \frac{S_{X, Z}  S_{Y, Z} S^{*}_{Z, W}}{S_{V, Z}} \ch[W]  \right)  .
\end{split}
\end{equation}
\begin{corollary}\label{cor:Ver} 
We keep the above set-up, in particular we retain assumption \ref{assumption}, the ones of Definition \ref{def:Verss2} and assume that
the $S$-kernel is unitary.
 
Then the Verlinde formula holds for standard modules, that is 
\[
\ch[X] \times_V \ch[Y]  = \ch[X\otimes_V Y] 
\]  for all standard modules $X$ and all obejcts $Y$ in $\cC$. 
\end{corollary}
\begin{proof} 
Let $X$ be a standard module and $Y$ be any object in $\cC$. 
By  Corollary \ref{cor:typprod} the tensor product $X \otimes_V Y$ is a standard module. 
By Theorem \ref{thm:Ver} $\cC$ admits a Verlinde algebra of quantum dimensions 
\[
\frac{S_{X, Z}}{S_{V, Z}}  \frac{S_{Y, Z} }{S_{V, Z}} = \frac{S_{X \otimes_V Y, Z}}{S_{V, Z}}.
\]
Hence
\begin{equation}\nonumber
\begin{split}
\ch[X] \times_V \ch[Y] &= \int_{T} \mu(Z) \left(\int_{T} \mu(W)\   \frac{S_{X, Z}  S_{Y, Z} S^{*}_{Z, W}}{S_{V, Z}} \ch[W]  \right)  \\
&= \int_{T} \mu(Z) \left(\int_{T} \mu(W)\   S_{X \otimes_V Y, Z}   S^{*}_{W, Z} \ch[W]  \right)  =   \ch[X \otimes_V Y].
\end{split}
\end{equation}
\end{proof}

\section{Examples}

In this section we discuss the two main examples of the logarithmic Verlinde formula in the literature, namely the cases of the singlet algebras and of the affine VOA of $\sltwo$ at admissible levels. In order to do so we first need to discuss some semisimple examples. 

\subsection{Semisimple Examples}\label{sec:examples}

\subsubsection{The Virasoro algebra}

Let $V= \Vir_{k}$ be the simple Virasoro vertex algebra at central charge $c_k = 13 - \frac{6}{k+2} - 6(k+2)$ and let $k = -2 + \frac{u}{v}$ be a non-degenerate principal admissible level for $\sltwo$, that is $u, v \in \ZZ_{>1}$ and $(u, v) =1$. Then $\Vir_k$ is rational \cite{W} and its simple modules are denoted by $L_{r, s}$ with $1 \leq r \leq u-1, 1 \leq s \leq v-1$ and $L_{r, s} \cong L_{r', s'}$ if and only if either $(r', s') = (r, s)$ or $(r', s') = (u-r, v-s)$. Set 
$\vfuscoeff{(r,s) (r',s')}{(r'',s'')} =  \vpfuscoeff{r, r'}{r''}{u} \vpfuscoeff{s, s'}{s''}{v}$,
where
\begin{equation} \label{eqn:VirFusCoeff}\nonumber
\vpfuscoeff{t, t'}{t''}{w} = 
\begin{cases}
1 & \text{if \(\abs{t-t'}+1 \leqslant t'' \leqslant \min \set{t+t'-1,2w-t-t'-1}\) and \(t+t'+t''\) is odd,} \\
0 & \text{otherwise.}
\end{cases}
\end{equation}
Then the Virasoro fusion rules are given by the formula
\[
L_{r, s} \otimes L_{r', s'} \cong \bigoplus_{r'' =1}^{u-1} \bigoplus_{s'' =1}^{v-1}\vfuscoeff{(r,s) (r',s')}{(r'',s'')} L_{r'', s''}
\]
The modular $S$-matrix coefficients are 
\begin{equation} \label{eqn:SVir}\nonumber
S_{(r,s) (r',s')} = -2 \sqrt{\frac{2}{uv}} \brac{-1}^{rs'+r's} \sin \frac{v \pi r r'}{u} \sin \frac{u \pi s s'}{v}
\end{equation}
and Verlinde's formula in this instance is
\[
\vfuscoeff{(r,s) (r',s')}{(r'',s'')} = \sum_{(R, S)} \frac{S_{(r,s) (R,S)} S_{(r',s') (R,S)} S^*_{(r'',s'') (R,S)}}{S_{(1, 1) (R,S)}}
\]
where the sum is over the set that labels inequivalent simple modules. 

\subsubsection{The Heisenberg VOA and their extensions}

The vertex tensor category structure of the Heisenberg VOA has been first stated as a Theorem in \cite[Thm.2.3]{CKLR}. By now \cite[Thm. 3.6]{CY}\cite[Thm.2.3]{Mc-mirror} immediately guarantee this. 
Consider a rank $n$ Heisenberg vertex algebra and choose generators $h_1, \dots, h_n$. Then their OPE
\[
h_i(z)h_j(w) = \frac{M_{i, j}}{(z-w)^2}
\]
defines a symmetric $n \times n$-matrix $M$, which we assume to be non-degenerate. We then denote the Heisenberg VOA by $\pi^M$ to indicate our choice, but of course $\pi^M \cong \pi^{M'}$ for any other non-degenerate symmetric $n \times n$-matrix $M'$. 
Let $v_1, \dots, v_n$ be a standard basis of $\CC^n$ with inner product $(v_i, v_j) = \delta_{i, j}$. Then we define the $\CC^n$-valued Heisenberg field $h(z) = h_1(z)v_1 + \dots + h_n(z)v_n$. Fix an element $b \in \CC^n$, then $b$ defines a Virasoro field
\[
L(z) = \frac{1}{2}:(h(z), M^{-1}h(z)): + \left(b, M^{-1} \frac{d}{dz}h(z)\right)
\]
of central charge
$c = n -12(b, M^{-1}b)$.
The simple modules of $\pi^M$ are Fock modules. 
Let $\lambda \in \CC^n$. Then the Fock module $\pi^M_\lambda$ is generated by a highest-weight vector $v_\lambda$ on which the zero-mode of $h_i$ acts by multiplication with $\lambda_i$. The conformal weight of the top level of $\pi^M_\lambda$ is
\[
h_\lambda = \frac{1}{2}(\lambda, M^{-1} \lambda) + (b, M^{-1} \lambda). 
\]
The fusion product of two Fock modules is $\pi^M_\lambda \otimes \pi^M_\mu \cong \pi^M_{\lambda+\mu}$, which follows from \cite{FZ92} as $\pi^M$ is the affine vertex algebra of the abelian Lie algebra $\CC^n$ at non-degenerate level. The Fock modules are simple and the only extensions that they admit are self-extensions. Since any Fock module is $C_1$-cofinite, since any self-extension is of finite length and since they are closed under the contragredient dual, the category of $C_1$-cofinite modules of $\pi^M$ form a vertex tensor category \cite[Thm. 3.6]{CY}\cite[Thm.2.3]{Mc-mirror}. We are only interested in the semisimple vertex tensor subcategory of Fock modules, which we now denote by $\cC^\pi$, i.e. we don't consider the self-extensions. This category is just $\Vect_{\CC^n}^Q$ where $Q = ( \cdot, M^{-1} \cdot)$. In particular this means that the braiding of $\pi^M_\lambda \otimes \pi^M_\mu$ is just $e^{\pi i (\lambda, M^{-1} \mu)}$ times the identity on $ \pi^M_{\lambda+\mu}$. The choice of $b$ is really a choice of duality structure on $\cC^\pi$, see \cite{Allen:2021kay}, and if we set $b=0$, then this duality is in fact a ribbon structure. In particular in this case one can take the trace of the double braiding, i.e. the Hopf link and this is just $e^{2\pi i (\lambda, M^{-1}\mu)}$. This can also be recovered from modular transformations of characters.  
The character of a Fock module is
\begin{equation}\nonumber
\ch[\pi^M_\lambda](q, z) = \tr_{\pi_M}\left( q^{L_0 - \frac{c}{24} e^{2\pi i (u, h_0)}}\right) = \frac{q^{\frac{1}{2}(\lambda+b, M^{-1}(\lambda+b))} z^\lambda}{\eta(q)^n}
\end{equation}
with $z^\lambda = e^{2\pi i (u, \lambda)}$ and $u \in \CC, \tau \in \HH$.
We assume that $M$ is positive definite. The general Gaussian integral says that for such a matrix $M$ and $d \in \CC^n$
\[
\int_{\RR^n} e^{-\frac{\alpha}{2}(w, M^{-1} w) + (d, w)} d^nw = \sqrt{\frac{(2\pi)^n}{\alpha \det M^{-1}}} e^{\frac{1}{2\alpha}(d, Md)}
\]
for $\alpha \in \CC$ with $\Re(\alpha) > 0$. Note the modular transformation of the Dedekind eta-function $\sqrt{-i\tau} \eta(\tau) = \eta(-1/\tau)$. Assume that $b \in \RR^n$. From this one computes
\begin{equation}\nonumber
\begin{split}
\int_{\RR^n} \ch[\pi^M_\lambda](q, z) e^{-2\pi i (\lambda + b, M^{-1}(\mu +b))} d\lambda &= 
\frac{e^{-2\pi i (u, b)}}{\eta(\tau)^n} \int_{\RR^n} e^{\pi i \tau (\lambda+b, M^{-1}(\lambda+b))} e^{2\pi i (\lambda + b,  u + M^{-1}(\mu + b))} d\lambda  \\
&=\frac{e^{-2\pi i (u, b)}}{\eta(-1/\tau)^n}  \sqrt{\det M} e^{-\frac{\pi i }{\tau}\left(u-M^{-1}(\mu+b), Mu-(\mu+b) \right)} \\
&=  e^{-\frac{\pi i }{\tau} (u, Mu)} e^{-2\pi i (u, b)} e^{\frac{2\pi i}{\tau} (u, b)}       \sqrt{\det M}     \ch[\pi^M_\mu]\left( -\frac{1}{\tau}, \frac{u}{\tau} \right) 
\end{split}
\end{equation}
so up to the usual automorphy factor $e^{-\frac{\pi i }{\tau} (u, Mu)} e^{-2\pi i (u, b)} e^{2\pi i (u, b)}$ for the Jacobi-like variable $z$ the modular $S$-transformation is described by the $S$-kernel 
\[
S_{\lambda, \mu} := \frac{e^{-2\pi i (\lambda + b, M^{-1}(\mu +b))}}{ \sqrt{\det M} }.
\]
The above transformation formula holds for any $\lambda \in \RR^n, \mu \in \CC^n$. The quantum dimension is 
\[
q_\lambda(\rho) := \frac{S_{\lambda, \rho}}{S_{0, \rho}}  = e^{-2\pi i (\lambda, M^{-1}(\rho +b))}
\]
Let $\cC^\pi_\RR$ be the subcategory of Fock modules that have real weight. 
Let $Q$ be the linear span of the functions $\{q_\lambda(\rho) | \lambda \in \RR^n\}$. Clearly the $q_\lambda(\rho) $ are linearly independent. 
$Q$ satisfies the quantum dimension algebra
\begin{equation}\nonumber
\begin{split}
q_\lambda(\rho)q_\mu(\rho) =\frac{S_{\lambda, \rho}}{S_{0, \rho}}  \frac{S_{\mu, \rho}}{S_{0, \rho}}  =  e^{-2\pi i (\lambda + \mu, M^{-1}(\rho +b))} = \frac{S_{\lambda+\mu, \rho}}{S_{0, \rho}} = q_{\lambda+\mu}(\rho)
\end{split}
\end{equation}
that is for each $\rho$, the normalized $S$-kernel defines a one-dimensional representation of the tensor ring of $\cC^\pi_\RR$, i.e. the map
$\pi^M_\lambda \mapsto q_\lambda(\rho)$ is a ring homomorphism from $K(\cC^\pi_\RR)$ to $Q$. We call $Q$ the algebra of quantum dimensions. 
We can analytically continue the $q_\lambda$ for all $\lambda \in \CC^n$ by setting $q_\lambda(\rho) = e^{-2\pi i (\lambda, M^{-1}(\rho +b))}$ and then our ring isomorphism extends to one between the complete Grothedieck ring and the analytically extended algebra of quantum dimensions.

For each $\lambda$, the $S$-kernel $q_\lambda(\rho)= \frac{S_{\lambda, \rho}}{S_{0, \rho}}$ defines a function in the variable $\rho$ and these are clearly linear independent and so the map $\pi^M_\lambda \mapsto q_\lambda(\rho)$ is an isomorphism of algebras.

 Let now both $\lambda, \mu \in \RR^n$ and fix $z \in \CC^n, q \in \HH$ and view the characters 
$\ch[\pi^M_\lambda] = \ch[\pi^M_\lambda](q, z)$ as functions in the variable $\lambda$. Since $q$ is in the unit disc these are clearly $L^1$-integrable functions. The $S$-kernel is unitary in the following sense
\begin{equation}\nonumber
\begin{split}
\int_{\RR^n} d\nu \left(\int_{\RR^n} d\mu\   S^{*}_{\lambda, \nu}  S_{\nu, \mu} \ch[\pi^M_\mu]  \right)  
&= \frac{1}{\det M}
\int_{\RR^n} d\nu \left(\int_{\RR^n} d\mu \  {e^{2\pi i (\nu + b, M^{-1}(\mu -\lambda))}} \ch[\pi^M_\mu]  \right) \\
&=  \frac{1}{\sqrt{\det M}}
\int_{\RR^n} d\nu \   F(\ch[\pi^M_{-(\nu+b)}]) e^{-2\pi i (\nu +b, M^{-1}\lambda) } \ = \ \ch[\pi^M_{\lambda}] \\ 
\int_{\RR^n} d\nu\ \left(\int_{\RR^n} d\mu   S_{\lambda, \nu}  S^{*}_{\nu, \mu} \ch[\pi^M_\mu]  \right)  &= \frac{1}{\det M}
\int_{\RR^n} d\nu \ \left(\int_{\RR^n} d\mu   {e^{-2\pi i (\nu + b, M^{-1}(\mu -\lambda))}} \ch[\pi^M_\mu]  \right) \\
&=  \frac{1}{\sqrt{\det M}}
\int_{\RR^n} d\nu\     F(\ch[\pi^M_{\nu+b}]) e^{-2\pi i (\nu +b, M^{-1}\lambda)} \  = \ \ch[\pi^M_{\lambda}] 
\end{split}
\end{equation}
with $F$ the Fourier transform, 
$F(f)(\mu) = \int_{\RR^n} d\nu\ f(\nu) e^{-2\pi i \mu \nu}$.
The Verlinde algebra of characters follows,
\begin{equation}\nonumber
\begin{split}
\ch[\pi^M_\lambda] \times_V \ch[\pi^M_\mu] &=  \int_{\RR^n} d\rho \left(\int_{\RR^n} d\nu\   \frac{S_{\lambda, \rho}  S_{\mu, \rho} S^{*}_{\nu, \rho}}{S_{0, \rho}} \ch[\pi^M_\nu]  \right) \\
&= \int_{\RR^n} d\rho \left(\int_{\RR^n} d\nu \  {e^{-2\pi i (\rho + b, M^{-1}(\lambda+ \mu -\nu))}} \ch[\pi^M_\mu]  \right)
= \ch[\pi^M_{\lambda+\mu}]. 
\end{split}
\end{equation}
Next we turn to a case where $M$ is still non-degenerate, but not positive definite anymore. This vertex algebra is called $\Pi(0)$, we reformulate \cite{Ad1}.

We start with the rank two Heisenberg vertex algebra associated to the matrix $M = \binom{0 \ 2}{ 2 \ 0}$ and denote the generators by $c(z), d(z)$. One also introduces the fields $\mu, \nu$ defined via $c = \frac{2}{k}(\mu-\nu), d =\mu +\nu$ and here $k \in \CC\setminus \{0\}$ is a parameter that will be used to define a family of Virasoro fields on $\Pi(0)$. The fields $\mu, \nu$ are orthogonal on each other and have OPEs
\[
\mu(z)\mu(w) = \frac{k/2}{(z-w)^2}, \qquad \nu(z)\nu(w) = -\frac{k/2}{(z-w)^2}.
\]
The conformal vector is chosen to be
\[
L = \frac{1}{2} :c(z)d(z): -\frac{1}{2}\frac{d}{dz} d(z) + \frac{k}{4}\frac{d}{dz} c(z).
\]
It has central charge $c = 6k+2$.  
Let $\pi_{a, b}$ be the Fock module on which the zero-mode of $c$ acts by multiplication with $a$ and the one of $d$ by multiplication with $b$. $\Pi(0)$ is then defined as the simple current extension
\[
\Pi(0) = \bigoplus_{n \in 2\ZZ} \pi_{0, n}.
\]
Let $\cD$ be the category of $\Pi(0)$-modules that are objects in the direct limit completion of $\cC_\pi$, see \cite{CMY} for the notion of direct limit completions of vertex tensor categories. This category can be described using vertex algebra extensions \cite{HKL, CKL, CKM, CMY}, which is particularly simple for simple current extensions. See \cite{CMY3} for a comparable example. The point is that there is an induction functor $\cF$, which maps the Fock-module $\pi_{a, b}$ to the not necessarily local $\Pi(0)$-module
\[
\bigoplus_{n \in 2\ZZ} \pi_{a, b+n}.
\]
Locality is decided by the monodromy with the $\pi_{0, n}$. This monodromy is $e^{\pi i a n}$ and so one gets a local module for any $b \in \CC$ while $a$ needs to be an integer.
It is convenient for applications to label these modules
\[
\Pi_\ell(\lambda) = \bigoplus_{n \in 2\ZZ} \pi_{a, b+n}, \qquad a = \ell,  b =2\lambda + \frac{k}{2}.
\]
The induction functor $\cF$ is monoidal and hence the fusion rules
\[
\Pi_\ell(\lambda)  \otimes \Pi_{\ell'}(\lambda') = \Pi_{\ell+\ell'}(\lambda+ \lambda')  
\]
hold. Also note the periodicity $\Pi_\ell(\lambda+1) \cong \Pi_\ell(\lambda)$. 
Set $h = -2\mu$\footnote{Here we choose a different sign then \cite{Ad1}. The reason is that with our convention formulae match with \cite{CreMod12, CR1}} and define the character of $\Pi_\ell(\lambda)$ as usual. It is computed in \cite[Prop. 4.5]{Ad1} for $\ell =-1$ and the general computation is basically the same. 
\[
\ch[\Pi_{-\ell-1}(\lambda)] = \tr_{\Pi_{-\ell-1}(\lambda)}(q^{L_0-\frac{c}{24}} z^{h_0}) = q^{\frac{\ell}{2}\left(k-2\lambda\right)+ \frac{k\ell^2}{4}} z^{k-2\lambda+k\ell} \frac{\delta(z^2q^\ell -1)}{\eta(q)^2}
\]
with $\delta(x-1)  = \sum_{n \in \ZZ} x^n$ the formal delta distribution. Assume that $k = -2 + \frac{u}{v}$ for $u, v \in \ZZ_{>1}$ co-prime. 
This character relates to the character of the module $\sigma^\ell(\TypMod{\lambda, \Delta_{r, s}})$ of the affine VOA of $\sltwo$ at admissible level $k$ as 
\[
\sigma^\ell(\TypMod{k-2\lambda, \Delta_{r, s}}) = \ch[\Pi_{-\ell-1}(\lambda)]  \ch[L^k_{r, s}].
\]

Modular transformations of this is quite subtle and were obtained for the $\sigma^\ell(\TypMod{k-2\lambda, \Delta_{r, s}}) $ in  \cite[Section 3.1]{CreMod12} \cite[Thm. 6]{CR1} by interpreting the delta distribution in a suitable way (recall that $z=e^{2\pi i u}, q = e^{2\pi i \tau}$)
\[
\delta(z^2q^\ell -1) = \sum_{m\in\ZZ} \delta(2u+\ell\tau -m)
\]
where $\delta(2u+\ell\tau -m)$ is a formal distribution satisfying $\delta(2u+\ell\tau -m)f(u, \tau) = \delta(2u+\ell\tau -m)f((m-\ell\tau)/2, \tau)= \delta(2u+\ell\tau -m)f(u, (m-2u)/\ell)$ (the last identity of course only if $\ell \neq 0$). In this way one interprets $\ch[\pi_\ell(\lambda)]$ as a power series in formal distributions with coefficients functions on $\tau, u$ and the parameters $\ell, \lambda$. Explicitly
\begin{equation}\label{eq:chpi} \nonumber
\begin{split}
\ch[\Pi_{-\ell-1}(\lambda)] &= q^{\frac{\ell}{2}\left(k-2\lambda\right)+ \frac{k\ell^2}{4}} z^{k-2\lambda+k\ell} \frac{\delta(z^2q^\ell -1)}{\eta(q)^2} 
= \frac{q^{\frac{\ell}{2}\left(k-2\lambda\right)+ \frac{k\ell^2}{4}} z^{k-2\lambda+k\ell}}{\eta(q)^2} \sum_{m \in \ZZ}\delta(2u+\ell\tau -m) \\
&= \frac{q^{ \frac{k\ell^2}{4}} z^{k\ell} }{\eta(q)^2} \sum_{m \in \ZZ}   q^{\frac{\ell}{2}\left(k-2\lambda\right)}   z^{k-2\lambda} \delta(2u+\ell\tau -m) 
= \frac{q^{ \frac{k\ell^2}{4}} z^{k\ell} }{\eta(q)^2} \sum_{m \in \ZZ}   e^{\pi i m (k-2\lambda)} \delta(2u+\ell\tau -m).
\end{split}  
\end{equation}
Modular properties of these distributions are studied in \cite{CR1}. One gets a projective  $SL(2, \ZZ)$-action due to the appearance of a phase corresponding to the argument of $\tau$. This phase however cancels in quantum dimensions and Verlinde's formula and so we will ignore it. Translated to our notation  \cite[Thm. 6]{CR1}  gives the $S$-kernel and normalized $S$-kernel, i.e. quantum dimension,
\[
S_{\Pi_\ell(\lambda), \Pi_{\ell'}(\lambda')} =  e^{-i \pi (k \ell\ell'-k +2\lambda'(\ell+1) +2\lambda(\ell'+1))} \qquad \text{and} \qquad  q_{\ell,\lambda}(\ell', \lambda') = 
\frac{S_{\Pi_\ell(\lambda), \Pi_{\ell'}(\lambda')}}{S_{\Pi(0), \Pi_{\ell'}(\lambda')}} = e^{-i \pi (k \ell\ell'+2\lambda' \ell +2\lambda(\ell'+1))}
 \]
The algebra of quantum dimension holds 
\begin{equation}
\begin{split}
q_{\ell,\lambda}(\ell', \lambda')  q_{m,\mu}(\ell', \lambda') 
 &= e^{-i \pi (k \ell\ell'+2\lambda' \ell +2\lambda(\ell'+1))}e^{-i \pi (k m\ell' +2\lambda' m+2\mu(\ell'+1))} \\ & = e^{-i \pi (k (\ell+m)\ell' +2\lambda' (\ell+m) +2(\lambda+\mu)(\ell'+1))} = q_{\ell + m,\lambda + \mu}(\ell', \lambda') 
\end{split}
\end{equation}
and we get a Verlinde formula by observing that with our interpretation \eqref{eq:chpi} the characters are power series with coefficients functions in $u, \tau, \lambda, \ell$ and they are one-periodic in $\lambda$. With this interpretation one can compute the unitarity of the $S$-kernel and the Verlinde algebra of characters. We only   compute the Verlinde algebra of characters
\begin{equation}\nonumber
\begin{split}
\ch[\Pi_\ell(\lambda)] \times_V \ch[\Pi_m(\mu)] &= \sum_{n\in \ZZ} \int_0^1 d\nu \sum_{r \in \ZZ} \int_0^1 d\rho\ \frac{ S_{\Pi_\ell(\lambda), \Pi_r(\rho)} S_{\Pi_m(\mu), \Pi_r(\rho)} S_{\Pi_n(\nu), \Pi_r(\rho)}^{*}}{S_{\Pi(0), \Pi_r(\rho)}} \ch[\Pi_n(\nu)] \\
&= \sum_{n\in \ZZ} \int_0^1 d\nu \sum_{r \in \ZZ} \int_0^1 d\rho\    e^{- i\pi(kr(\ell+m-n) +2\rho(\ell+m-n) +2(\lambda+\mu-\nu)(r+1))}      \ch[\Pi_n(\nu)] \\
&= \sum_{n\in \ZZ} \int_0^1 d\nu \sum_{r \in \ZZ} \  \delta_{\ell+m, n}    e^{- 2\pi i(\lambda+\mu-\nu)(r+1)}      \ch[\Pi_n(\nu)] \\
&= \sum_{n\in \ZZ} \int_0^1 d\nu  \  \delta_{\ell+m, n}  \delta(\lambda+ \mu-\nu)  e^{- 2\pi i(\lambda+\mu-\nu)}      \ch[\Pi_n(\nu)] \\
&= \sum_{n\in \ZZ}   \  \delta_{\ell+m, n}     \ch[\Pi_n(\lambda+\mu)] 
=    \ch[\Pi_{\ell+m}(\lambda+\mu)]. 
\end{split}
\end{equation}
Here we used that the delta-distribution for one-periodic functions is 
$\delta(x) = \sum\limits_{r \in \ZZ} \    e^{2\pi i rx}$.

\subsection{Non-semisimple Examples}

Presently there aren't many examples of non-semisimple categories of VOAs that are well-understood and here the singlet algebras and the affine VOA of $\sltwo$ at admissible levels will be discussed. 
Besides those, the well-studied VOAs of non semi-simple type are the triplet algebras  \cite{AdaTri08, TW}, the affine VOA of $\mathfrak{gl}_{1|1}$ \cite{CreRel11,CMY1}, the $\beta\gamma$-VOA \cite{Allen:2020kkt} and the $\mathcal B_p$-algebras \cite{CreCos13, CMY3}. All these are simple current extensions of a singlet algebra, respectively of a singlet algebra times a Heisenberg VOA. 

In our following discussion, we have to restrict weights so that expressions converge. It turns out that the $S$-kernel for resolutions converge as long as the imaginary part of weights is not positive and then for integration against $S$-kernels we in addition need to require (as before) that weights are real.

\subsubsection{Singlet Algebras}

The singlet algebras $\cM(p)$ for $p \in \ZZ_{\geq 2}$ have the realization $\cM(p) \hookrightarrow \pi$ with $\pi$ a rank one Heisenberg VOA. 
Note that the conformal vector of the singlet algebra coincides with a non-standard conformal vector of the singlet algebra, see e.g. \cite{CreFal13}.
This algebra is studied in \cite{Ad-singlet, CreFal13, CMY4, CMY6, CLR}. In particular the existence of vertex tensor category, rigidity, being locally finite and having enough projectives are all shown in \cite{CMY4, CMY6}.  Fock modules of the Heisenberg VOA allow for two interesting vertex tensor categories. We will consider the category $\cD$ whose objects are direct sums of Fock modules of real weight. One can enlarge this category to the category of finite-length $C_1$-cofinite modules, which then also includes self-extensions of arbitrary length of Fock modules. Correspondingly the singlet algebra has two categories and we consider the category $\cC$ whose analogue for complex weights  is denoted by $\mathcal O^T(\cM(p))$ in \cite{CMY6}. 
Maybe the simplest precise way to describe $\cC$ is as the category that is equivalent to the category of real weight modules of the small unrolled quantum group of $\sltwo$ at $q = e^{\frac{\pi i}{p} }$ via the logarithmic Kazhdan-Lusztig correspondence of \cite{CLR}.
This is precisely the category, so that $\cC_\pi^\text{loc}\cong \cD$. We already discussed that $\cD$ is semisimple and admits a semisimple Verlinde algebra of quantum dimensions.
That $\Irr(\cC_A) = \Irr(\cD)$ is shown in \cite{CLR}. The remaining assumptions are easily extracted from the literature. Firstly the Fock modules are parameterized by its highest-weight, a complex number $\lambda$. We denote this by $\pi_\lambda$ and write $\cG(\pi_\lambda) = F_\lambda$. The $F_\lambda$ are the basic standard modules. Set $\alpha_+ = \sqrt{2p}, \alpha_- = -\sqrt{2/p}$, $\alpha_0 = \alpha_+ + \alpha_-$. Set 
\[
\alpha_{r, s} = \frac{1-r}{2}\alpha_+ + \frac{1-s}{2}\alpha_-, \qquad F_{r, s} := F_{\alpha_{r, s}}, \qquad \pi_{r, s} := \pi_{\alpha_{r, s}}
\]
for $r \in \ZZ$ and $s =1, \dots, p$. The modules $F_\lambda$ for $\lambda \notin \{ \alpha_{r, s} | r \in \mathbb Z,  s = 1 , \dots,  p-1 \}$ are projective, injective and simple and hence the same is true for their duals. Otherwise they satisfy a non-split exact sequence
\[
0 \rightarrow M_{r-1, p-s} \rightarrow F_{r-1, p-s} \rightarrow M_{r, s} \rightarrow 0
\]
and $M_{r, s}$ for $r \in \ZZ$ and $s = 1, \dots, p-1$ are the only simple modules that are not standard. From this it is clear that basic standard modules are linearly independent. In addition the map $\tau: \Irr(\cC) \rightarrow \Irr(\cD)$ is defined by $\tau(F_\lambda) = \pi_\lambda$ on simple standard modules $F_\lambda$ and $\tau(M_{r, s}) = \pi_{r-1, p-s}$. 
The singlet VOA is the module $\cM(p) = M_{1, 1}$ and the algebra object is $A = F_{1, 1} = \cG(\pi)$. Splicing the exact sequences one gets a resolution in terms of standard modules
\begin{equation}\label{eq:res-singlet}
\cdots \rightarrow F_{r-5, p -s }  \rightarrow F_{r-4, s }  \rightarrow F_{r-3, p -s }  \rightarrow F_{r-2, s }  \rightarrow F_{r-1, p -s }  \rightarrow M_{r, s } \rightarrow 0.
\end{equation}
Since $\pi_\alpha \otimes_\pi \pi_\beta \cong \pi_{\alpha+\beta}$, it is obvious that the category of Fock modules with real weight is $\mathbb R$-graded and 
\[
\rho: \mathbb R \rightarrow \mathbb R, \qquad \lambda \mapsto -\frac{2}{\alpha_-} \lambda
\]
is a group homomorphism for which $F_{{r, s}}$ has degree $s-1-p(r-1)$. In particular 
the degree of the $n$-th term of the chain complex corresponding to \eqref{eq:res-singlet}
has degree $-(s+1)+p(n+3-r)$ if $n$ is even and $s-1+p(n+2-r)$ if $n$ is odd, that is these are strict $\rho$-ordered standard resolutions. 

The singlet algebra has more modules, denoted by $\overline{F}_{r, s}$ and fitting into the non-split exact sequence (see equation (34) of \cite{CLR})
\[
0 \rightarrow M_{r, s} \rightarrow \overline{F}_{r, s} \rightarrow M_{r-1, p-s} \rightarrow 0.
\]
Point (9) of Lemma 8.5 of \cite{CLR} for $\lambda = \alpha_{2, p-1} =\alpha_-$ says that $\cF(\overline{F}_{2, p-1})$ is a single indecomposable module with Loewy diagram
\[
\pi_0 \rightarrow \pi_{\alpha_-} \rightarrow \pi_{2\alpha_-} \rightarrow \cdots \rightarrow \pi_{(p-1)\alpha_-}.
\]
Since induction is exact (since $\cC$ is rigid) and since $\cF(M_{1, 1}) = \pi_0$ it follows that $\cF(M_{2, p-2})$ has Loewy diagram
\[
 \pi_{\alpha_-} \rightarrow \pi_{2\alpha_-} \rightarrow \pi_{3\alpha_-} \rightarrow \cdots \rightarrow \pi_{(p-1)\alpha_-}.
\]
Frobenius reciprocity says that
\[
\Hom_{\cC_A}(\cF(F_{1, 1}), \pi_{r, s}) = \Hom_\cC(F_{1, 1}, F_{r, s}) =  \begin{cases} \CC & \quad (r, s) \in \{ (1, 1), (2, p-1) \\ 0 & \quad \text{otherwise} \end{cases}
\]
since $F_{1, 1}$ has $M_{1, 1}$ as submodule and $M_{2, p-1} \cong F_{1, 1}/M_{1, 1}$ and by exactness of induction the only possibility is that $N= \cF(F_{1, 1}) = \pi_0 \oplus \cF(M_{2, p-1})$. 

Due to the unusual conformal vector (see \cite{CreFal13} for details), the $S$-kernel of the Heisenberg VOA is
$S_{\lambda, \mu} := S_{\pi_\lambda, \pi_{\mu}} = e^{-2\pi i \left( \lambda - \frac{\alpha_0}{2}\right)\left( \mu - \frac{\alpha_0}{2}\right)}$,
in particular
$S_{(r, s), \mu} := S_{\pi_{r, s}, \pi_\mu} = e^{-2\pi i (r\alpha_+ + s\alpha_-)\left( \mu - \frac{\alpha_0}{2}\right) }$.
Set $\zeta_\mu := e^{- \pi i \alpha_+\mu}$, so that 
$q^A_{(r-1-i, p-s)}(\mu + \frac{\alpha_0}{2}) = \zeta_\mu^i \zeta_\mu^{3-r} \zeta_\mu^{-\frac{s+1}{p}}$ and
$q^A_{(r-1-i, s)}(\mu + \frac{\alpha_0}{2}) = \zeta_\mu^i \zeta_\mu^{2-r} \zeta_\mu^{\frac{s-1}{p}}$.
 Then for $|t| <1$ and recall that $\mu \in \mathbb R$ the resolution \eqref{eq:res-singlet} tells us that
\begin{equation}\nonumber
\begin{split}
q^A_{M_{r, s}\bullet}(\mu + \frac{\alpha_0}{2})(t, z)
&=  \sum_{\substack{i = 0 \\ i \ \text{even}}}^\infty t^i z^{pi}   \zeta_\mu^{i+3-r -\frac{s+1}{p}}- \sum_{\substack{i = 0 \\ i \ \text{odd}}}^\infty t^i z^{2s + p(i-1)}  \zeta_\mu^{i+2-r +\frac{s-1}{p}} = \frac{\zeta_\mu^{3-r  -\frac{s+1}{p}}}{1 - \zeta_\mu^2 z^{2p}t^2} -  \frac{tz^{2s}\zeta_\mu^{3-r + \frac{s-1}{p}}}{1 - \zeta_\mu^2 t^2 z^{2p}} \\ &= \zeta_\mu^{3-r-\frac{1}{p}}  \frac{ \zeta_\mu^{-\frac{s}{p}} -  t z^{2s}  \zeta_\mu^{\frac{s}{p}}}{1 - \zeta_\mu^2 t^2 z^{2p}}. 
\end{split}
\end{equation}
Thus for $\mu \notin E := \frac{1}{\alpha_+}\mathbb Z$
\[
q^A_{M_{r, s}}(\mu + \frac{\alpha_0}{2}) = \lim_{t\rightarrow 1^-} \lim_{z\rightarrow 1^-}q^A_{M_{r, s}\bullet}(\mu + \frac{\alpha_0}{2})(t, z) = \lim_{z\rightarrow 1^-} \lim_{t\rightarrow 1^-}q^A_{M_{r, s}\bullet}(\mu + \frac{\alpha_0}{2})(t, z)
 =
 - \zeta_\mu^{2-r-\frac{1}{p}}
  \frac{\sin\left(\pi s \alpha_- \mu \right)}{\sin\left(\pi  \alpha_+ \mu \right)}.
\]
It follows that (recall that $V =\cM(p) = M_{1, 1}$)
\begin{equation}\nonumber
\begin{split}
q_{F_\lambda}(\mu+\frac{\alpha_0}{2}) = \frac{q^A_{F_\lambda}(\mu+\frac{\alpha_0}{2})}{q^A_{V}(\mu+\frac{\alpha_0}{2})} =  - \frac{\sin(\pi \alpha_+ \mu)}{\sin(\pi \alpha_- \mu)} e^{-2\pi i  \lambda\mu} \zeta_\mu^{\frac{1}{p}-1} \\
 q_{M_{r, s}}(\mu+\frac{\alpha_0}{2}) = \frac{q^A_{M_{r, s}}(\mu+\frac{\alpha_0}{2})}{q^A_{V}(\mu+\frac{\alpha_0}{2})} =
    \frac{\sin(\pi s\alpha_- \mu)}{\sin(\pi \alpha_- \mu)} \zeta_\mu^{1-r}
\end{split}
\end{equation}
These can be analytically continued from functions on $\mathbb R \setminus E$ to $\mathbb C \setminus E$.
 Set $\mu = i \epsilon$ and $q_\epsilon = e^{\pi \epsilon}$,
 then 
 \begin{equation}\label{eq:S-singlet-cat}
  q_{F_\lambda}(\mu+\frac{\alpha_0}{2}) = - \frac{\sin(\pi \alpha_+ \mu)}{\sin(\pi \alpha_- \mu)} q_\epsilon^{2\lambda- \alpha_0} \qquad \text{and} \qquad
   q_{M_{r, s}}(\mu+\frac{\alpha_0}{2})  =  \frac{\sin(\pi s\alpha_- \mu)}{\sin(\pi \alpha_- \mu)} q_\epsilon^{- \alpha_+ (r-1)}\ .
 \end{equation}
 
 Jointly with Antun Milas, we studied Verlinde's formula of the singlet algebra in the analytic setting as follows  \cite{CreFal13}.
 The characters of the simple modules $M_{r, s}$ are false theta functions. These don't have good modular properties and so we introduced a regularization parameter $\epsilon$ and studied modular transformations of these deformed characters. We were able to use those to define a ring structure, the Verlinde algebra, on the linear span of regularized characters and conjectured that this ring structure coincides with the one of the Grothendieck ring  \cite[Conjecture 26]{CreFal13}. 
 
 We then introduced regularized quantum dimensions 
  as the limit as the modular parameter $\tau$ approaches zero of the quotient of the regularized character of the module $M$ by the regularized one of the VOA $\cM(p)$, 
  \[
 \text{qdim}[M^\epsilon] = \lim_{\tau \rightarrow 0^+} \frac{\ch[M^\epsilon]}{\ch[M_{1, 1}^\epsilon]}.
 \]
 These are functions of the parameter $\epsilon$ and form an algebra under multiplication. We observed that this algebra is isomorphic to the Verlinde algebra that we defined \cite[Theorem 28]{CreFal13}.
 
 The regularized quantum dimensions are \cite[Proposition 27]{CreFal13}
\begin{equation}\nonumber
 \text{qdim}[F_\lambda^\epsilon]   = - \frac{\sin(\pi \alpha_+ \mu)}{\sin(\pi \alpha_- \mu)} q_\epsilon^{2\lambda- \alpha_0} \qquad \text{and} \qquad
  \text{qdim}[M_{r, s}^\epsilon] =  \frac{\sin(\pi s\alpha_- \mu)}{\sin(\pi \alpha_- \mu)} q_\epsilon^{- \alpha_+ (r-1)}
 \end{equation}
Comparing with \eqref{eq:S-singlet-cat} we observe that with $\mu= i\epsilon$
\[
 \text{qdim}[F_\lambda^\epsilon]   = q_{F_\lambda}(\mu+\frac{\alpha_0}{2}) \qquad \text{and} \qquad
  \text{qdim}[M_{r, s}^\epsilon] =  q_{M_{r, s}}(\mu+\frac{\alpha_0}{2}) .
 \]
 \begin{corollary}
 The Verlinde conjecture {\textup{\cite[Conjecture 26]{CreFal13}}} holds for the category $\mathcal C$ of $\mathcal M(p)$. 
 \end{corollary}

\begin{remark}\textup{
The category of weight modules of the singlet algebra and of the small unrolled quantum group of $\sltwo$, $u_{q}^H(\sltwo)$, at $q = e^{\frac{\pi i}{p}}$ are braided tensor equivalent \textup{\cite{CLR}}. The quantum group $u_{q}^H(\sltwo)$ has been studied in \textup{\cite{CGP-unrolled}}.
Open Hopf links on simple modules provide one-dimensional representations of the tensor ring in ribbon categories. Taking traces of open Hopf links then gives the usual Hopf links in the case of modular tensor categories. The trace on  almost all simple modules of $u_{q}^H(\sltwo)$ vanishes and one replaces the usual trace by a modified trace on the ideal of negligible objects. This has been done for $u_{q}^H(\sltwo)$ in
\textup{\cite{Creutzig:2016htk}} and the resulting modified Hopf links coincide with the regularized quantum dimension of \textup{\cite{CreFal13}}.}
\end{remark}

\subsubsection{The affine VOA of $\sltwo$ at admissible levels}
 Let $L_k(\sltwo)$ be the simple affine VOA of $\sltwo$ at level $k$.
$k$ is non-integral admissible if it is of the form
 $k = - 2 + \frac{u}{v}$ with $u,v \in \ZZ_{>1}$ co-prime. It has been studied extensively  \cite{AdaVer95, Ad1, RW2, RidFus10, Kawrel19, RidSL208, C23, ACK23, CR1, CreMod12, CMY24}. In particular its category of weight modules is a vertex tensor category \cite{C23}, it is rigid \cite{CMY24}, there is a complete classification of indecomposable modules that implies that the category of weight modules is locally finite and has enough projectives \cite{ACK23} and there is a good realization \cite{Ad1}.
Most importantly for this work, the standard formalism for Verlinde's formula has appeared in this case in \cite{CR1, CreMod12}. As in the singlet case, we restrict to the subcategory of real weight modules, denote it by $\mathcal C$. Since every intertwiner of $L_k(\sltwo)$-modules is in particular one of modules for the Heisenberg subalgebra corresponding to the Cartan subalgebra of $\sltwo$. Since real weight modules of the Heisenberg VOA close under tensor product the same must be true for $\mathcal C$.

Let us first list some modules. The generic modules $\sfmod{\ell}{\TypMod{\lambda ; \Delta_{r,s}}}$ are parameterized by triples $(\ell, \lambda, \Delta_{r, s})$ where $\ell \in \ZZ$ is called the spectral flow index. $\lambda \in \CC/\ZZ$ is a (relaxed-highest) weight and $(r, s)$ are integers with $1 \leq r \leq u-1, 1 \leq s \leq v-1$ and $\Delta_{r, s}$ is a conformal weight associated to these labels, that is  
\begin{equation} \label{eqn:DefConfDimRS}\nonumber
\Delta_{r,s} = \frac{\brac{vr-us}^2 - v^2}{4uv},
\end{equation}
The special weight labels are
\begin{equation} \label{eqn:DefLambda}\nonumber
\lambda_{r,s} = r-1- \frac{u}{v}s
\end{equation}
and there are a few relations
$\lambda_{u-r,v-s} = -\lambda_{r,s} - 2$ and $\Delta_{u-r,v-s} = \Delta_{r,s}$. 
If $\lambda \notin \{\lambda_{r, s}, \lambda_{u-r, v-s}\}$, then $\sfmod{\ell}{\TypMod{\lambda ; \Delta_{r,s}}}$ is simple, projective and injective and hence the same is true for its dual. 
If $\lambda = \lambda_{r, s}$, then there are two modules $\sfmod{\ell}{\TypMod{r,s}^+}$ and $\sfmod{\ell}{\TypMod{u-r,v-s}^-}$ satisfying the non split exact sequences
\begin{equation} \label{ES:Ers}
\begin{split}
&\dses{\sfmod{\ell}{\DiscMod{r,s}^\pm}}{\sfmod{\ell}{\TypMod{r,s}^\pm}}{\sfmod{\ell}{\DiscMod{u-r,v-s}^\mp}}.\qquad 
\end{split}
\end{equation}
The non-generic simple modules are denoted by $\sfmod{\ell}{\DiscMod{r,s}^\pm}$ and $\sfmod{\ell}{\IrrMod{r, 0}}$ and there are quite a few identifications
\begin{equation} \label{eqn:SFRelations} \nonumber
\sfmod{\ell+1}{\IrrMod{r,0}} = \sfmod{\ell}{\DiscMod{u-r,v-1}^+}, \ \ \ \sfmod{\ell-1}{\IrrMod{r,0}} = \sfmod{\ell}{\DiscMod{u-r,v-1}^-}, \ \ \ \sfmod{\ell-1}{\DiscMod{r,s}^+} = \sfmod{\ell}{\DiscMod{u-r,v-1-s}^-}, \ \ \ s \neq v-1.
\end{equation}
In particular a complete list of inequivalent simple non-generic modules is given by $\sfmod{\ell}{\DiscMod{r, s}^-}$. The affine VOA itself is $\IrrMod{1,0}=\sfmod{}{\DiscMod{u-1, v-1}^-}$. The short exact sequence \eqref{ES:Ers} for  $\sfmod{\ell}{\TypMod{r,s}^+}$  in terms of the $\sfmod{\ell}{\DiscMod{r, s}^-}$ is
\begin{equation} \label{ES:Ers}
\begin{split}
&\dses{\sfmod{\ell+1}{\DiscMod{u-r,v-s-1}^-}}{\sfmod{\ell}{\TypMod{r,s}^+}}{\sfmod{\ell}{\DiscMod{u-r,v-s}^-}}, \qquad s \neq v-1\\
&\dses{\sfmod{\ell+2}{\DiscMod{r,v-1}^-}}{\sfmod{\ell}{\TypMod{r,v-1}^+}}{\sfmod{\ell}{\DiscMod{u-r, 1}^-}},
\end{split}
\end{equation}
From this it is clear that basic standard modules are linearly independent.
Drazen Adamovic introduced the very important realization of $L_k(\sltwo)$, namely an embedding $\iota: L_k(\sltwo) \hookrightarrow A = \Vir_{k} \otimes \Pi(0)$.

The category of $\Pi(0)$ is a category of $\mathbb Z \times S^1$-graded vector spaces and in particular $\Vir_{k} \otimes \Pi(0)$-modules inherit this grading. Define the morphism $\rho$ by $\rho(L_{r, s} \otimes \Pi_\ell(\lambda)) = -\ell$, i.e. in particular every object is integer graded in this instance. 

Let $\omega$ be the automorphism of $\sltwo$ corresponding to the Weyl reflection. It lifts to an automorphism of the affinization of $\sltwo$ and also of $L_k(\sltwo)$. It satisfies $\omega \circ \sigma^\ell = \sigma^{-\ell} \circ \omega$ and if we consider modules twisted by $\omega$ then one gets the following relations
\[
\omega(\sfmod{\ell}{\DiscMod{r, s}^+}) = \sfmod{-\ell}{\DiscMod{r, s}^-}, \qquad \omega(\sfmod{\ell}{\TypMod{r,s}^-}) = \sfmod{-\ell}{\TypMod{r,s}^+}, 
\qquad \omega(\sfmod{\ell}{\TypMod{\lambda ; \Delta_{r,s}}}) = \sfmod{-\ell}{\TypMod{-\lambda ; \Delta_{r,s}}}
\]
We consider the embedding $\iota \circ \omega$ of $L_k(\sltwo)$.  In terms of the realization this means that we consider the embedding of $L_k(\sltwo)$ given by the formulae (compare with Proposition 3.1 of \cite{Ad1})
\begin{equation}\nonumber
\begin{split}
e(z) &\mapsto \left( (k+2)L(z) - :\nu(z)\nu(z): -(k+1) \frac{d}{dz}\nu(z) \right) e^{-\frac{2}{k}(\mu - \nu)}, \qquad
h(z) \mapsto -2\mu(z), \qquad
f(z) \mapsto e^{\frac{2}{k}(\mu-\nu)}.
\end{split}
\end{equation}
The identification of modules is
\[
\cG(L_{r, s} \otimes \Pi_{-\ell -1}(\lambda)) \cong  \begin{cases}
    \sigma^\ell(\cE_{-2\lambda+k, \Delta_{r, s}}) & \text{else} \\
    \sigma^\ell(\cE^+_{u-r, v-s}) & \lambda = \nu_{r, s} \\
    \sigma^\ell(\cE^+_{r, s}) & \lambda = \nu_{u-r, v-s} \\
\end{cases} 
\quad
\text{for} \quad 
\nu_{r, s} := \frac{1}{2}(r-1-\frac{u}{v}(s-1)).
\]
The modules $\sigma^\ell(\cD^-_{r, s})$ appear as both submodules and quotients of simple modules in $\cD$, e.g.
\[
\cG(L_{r, s} \otimes \Pi_{-\ell -1}(\nu_{r, s})) \cong \sigma^\ell(\cE^+_{u-r, v-s}) \twoheadrightarrow \sigma^\ell(\cD^-_{r, s}).
\]
Thus every simple module of $L_k(\sltwo)$ appears exactly once as the top of a simple module of the VOA $A$. 
That all irreducible objects in $\cC_A$ are local is proven in \cite{CMY24}. \newline
\indent The only reason that we use $\iota\circ \omega$  instead of $\iota$ is that \cite{CR1} used resolutions in terms of the $\sigma^{\ell}(\cE^+_{r, s})$ and we want to reproduce exactly those results.  
The resolution of the identity is obtained by splicing \eqref{ES:Ers}
\begin{equation} 
\begin{split} \label{eqn:ResL}
&\cdots \lra \sfmod{3v-1}{\TypMod{r,v-1}^+} \lra \cdots \lra \sfmod{2v+2}{\TypMod{r,2}^+} \lra \sfmod{2v+1}{\TypMod{r,1}^+} \\
&\qquad \qquad \lra \sfmod{2v-1}{\TypMod{u-r,v-1}^+} \lra \cdots \lra \sfmod{v+2}{\TypMod{u-r,2}^+} \lra \sfmod{v+1}{\TypMod{u-r,1}^+} \\
&\qquad \qquad\qquad\qquad\lra \sfmod{v-1}{\TypMod{r,v-1}^+} \lra \cdots \lra \sfmod{2}{\TypMod{r,2}^+} \lra \sfmod{}{\TypMod{r,1}^+} \lra \IrrMod{r,0} \lra 0.
\end{split}
\end{equation}
These resolutions are obviously strictly $\rho$-ordered.
The resolutions of other simple objects follow recursively from Corollary 9 of \cite{CR1}.
 The object $N$  is computed in \cite{CMY24}  it is $A \oplus M$, where $M$ fits into the non-split exact sequence
$0 \rightarrow L_{1, 1} \otimes \Pi_2(- \frac{u}{v}) \rightarrow M \rightarrow  L_{1, 2} \otimes \Pi_1(- \frac{u}{2v}) \rightarrow 0$.  

The resolution \ref{eqn:ResL} tells us that for $|t|, |z|  < 1$
\begin{equation}
\begin{split}
&q^A_{\IrrMod{r,0} \bullet}(\ell', \lambda', r', s') = e^{\pi i (k-2\lambda')}\sum_{s=1}^{v-1} (-1)^{s-1} \sum_{\ell = 0}^\infty \left(t^{s-1+2\ell(v-1)} z^{2v\ell+s}e^{-\pi i ((k\ell'+\lambda')(2v\ell+s) + \ell' \lambda_{r, s})} - \right. \\
& \qquad\qquad\qquad \left. t^{2(\ell+1)(v-1)-s}z^{2v(\ell+1)-s} e^{-\pi i ((k\ell'+\lambda')(2v(\ell+1)-s) - \ell' \lambda_{r, s})} \right) \frac{S^\text{Vir}_{(r, s), (r', s')}}{S^\text{Vir}_{(1, 1), (r', s')}} \\
&\qquad= \sum_{s=1}^{v-1} (-1)^{s-1}   \frac{ t^{s-1}z^s e^{-\pi i ((k\ell'+\lambda')s + \ell' \lambda_{r, s})} - 
 t^{2(v-1)-s}z^{2v-s} e^{-\pi i ((k\ell'+\lambda')(2v-s) - \ell' \lambda_{r, s})} }{1 - t^{2(v-1)} z^{2v} e^{-2\pi i v (k\ell'+\lambda')}} \frac{S^\text{Vir}_{(r, s), (r', s')}}{S^\text{Vir}_{(1, 1), (r', s')}}
\end{split}
\end{equation}
so this is a nice rational function in $t,z$. The limit $t, z \rightarrow 1^-$ doesn't depend on the order and is a rational function in trigonometric functions. Using Lemma 10 of \cite{CR1} this can be simplified as in the proof of Theorem 11 of \cite{CR1}. The answer is
\[
q^A_{\IrrMod{r,0}}(\ell', \lambda', r', s') =  \frac{e^{-\pi i \ell' (k+r-1)}}{2\cos(\pi \lambda') + 2(-1)^r \cos(k\pi s')} \frac{S^\text{Vir}_{(r, s), (r', s')}}{S^\text{Vir}_{(1, 1), (r', s')}}.
\]

Since all assumptions of our main Theorem hold, we have verified that
\begin{corollary}
The Verlinde formula of \textup{\cite{CR1, CreMod12}}  for the category of real weight modules of $L_k(\sltwo)$ holds.
\end{corollary}
It is not difficult to get the actual fusion rules from Verlinde's formula together with exactness of the tensor product. We demonstrate this for simple and projective modules.

Let $k$ be admissible  with $v>1$.  The Verlinde fusion rules of standard modules are  \cite[Proposition 13]{CR1}
\begin{multline} \label{eqn:GrFusTyp}\nonumber
\Gr{\sfmod{\ell}{\TypMod{\lambda ; \Delta_{r,s}}}} \fuse \Gr{\sfmod{\ell'}{\TypMod{\lambda' ; \Delta_{r',s'}}}} 
= \sum_{r'',s''} \vfuscoeff{(r,s) (r',s')}{(r'',s'')} \left(\Gr{\sfmod{\ell + \ell' + 1}{\TypMod{\lambda + \lambda' - k ; \Delta_{r'',s''}}}} + \right. \\
\left. \Gr{\sfmod{\ell + \ell' - 1}{\TypMod{\lambda + \lambda' + k ; \Delta_{r'',s''}}}} \right) 
+ \sum_{r'',s''} \brac{\vfuscoeff{(r,s) (r',s'-1)}{(r'',s'')} + \vfuscoeff{(r,s) (r',s'+1)}{(r'',s'')}} \Gr{\sfmod{\ell + \ell'}{\TypMod{\lambda + \lambda' ; \Delta_{r'',s''}}}}.
\end{multline}
In this and the following formulae $\Gr{\sfmod{\ell}{\TypMod{\lambda ; \Delta_{r,s}}}} =\Gr{\sfmod{\ell}{\TypMod{{r,s}}^+}}$ if $\lambda = \lambda_{r, s}$. Also recall that $\Delta_{r, s} = \Delta_{u-r, v-s}$. 
Since $\cC$ is rigid, the tensor product of a module with a projective module is projective. 
The standard projective modules are the $\sfmod{\ell}{\TypMod{\lambda; \Delta_{r, r}}}$,  for $\lambda \notin \{ \lambda_{r, s}   \lambda_{u-r, v-s}\}$
The projective modules at non-generic weight labels are denoted by $\sfmod{\ell}{\StagMod{r, s}}$ and satisfy the  relations 
$
\Gr{\sfmod{\ell}{\StagMod{r, s}}} = \Gr{\sfmod{\ell}{\TypMod{{r,s}}^+}} + \Gr{\sfmod{\ell+1}{\TypMod{{r,s+1}}^+}}
\ 
\text{for}  \ s \neq v-1\ \text{and} \  
\Gr{\sfmod{\ell}{\StagMod{r, v-1}}} = \Gr{\sfmod{\ell}{\TypMod{{r,v-1}}^+}} + \Gr{\sfmod{\ell+2}{\TypMod{{r,v-1}}^+}}
$ in $K(\cC)$.
Let $K^{\text{Proj}}(\cC)$ be the subring of $K(\cC)$ spanned by elements $[P]$ for $P$ projective in $\cC$.
Clearly two projective modules in $\cC$ are isomorphic if and only if they have the same standard modules (the $\sfmod{\ell}{\TypMod{{r,s}}^+}$) as composition factors. 
Thus there is a well-defined  map $P: K^{\text{Proj}}(\cC) \rightarrow \text{Obj}(\cC)$ assigning to each element in $K^{\text{Proj}}(\cC)$ its unique projective object. Explicitly:
\begin{equation}\nonumber
\begin{split}
&P: \Gr{\sfmod{\ell +  1}{\TypMod{\lambda  - k ; \Delta_{R, S}}}}  + \Gr{\sfmod{\ell}{\TypMod{\lambda ; \Delta_{R, S-1}}}} \mapsto  \sfmod{\ell}{\StagMod{R, S-1}},  \ \ \text{for} \ \lambda- k = \lambda_{R, S} \ \text{and} \ S \neq 1 \\
&P:  \Gr{\sfmod{\ell + 1}{\TypMod{\lambda - k ; \Delta_{R, 1}}}}  + \Gr{\sfmod{\ell -1}{\TypMod{\lambda ; \Delta_{R, 1}}}} \mapsto  \sfmod{\ell-1}{\StagMod{u-R, v-1}},  \ \ \text{for} \ \lambda - k = \lambda_{R, 1}  \\
&P: \Gr{\sfmod{\ell  -1}{\TypMod{\lambda + k ; \Delta_{R, S}}}}  + \Gr{\sfmod{\ell}{\TypMod{\lambda; \Delta_{R, S+1}}}} \mapsto  \sfmod{\ell}{\StagMod{R, S+1}},  \ \ \text{for} \ \lambda + k = \lambda_{R, S} \ \text{and} \ S \neq v-1  \\
&P: \Gr{\sfmod{\ell}{\TypMod{\lambda; \Delta_{R, S}}}} \mapsto \sfmod{\ell}{\TypMod{\lambda; \Delta_{R, S}}},  \ \qquad \text{for} \ \lambda \notin \{ \lambda_{R, S},   \lambda_{u-R, v-S}\}
\end{split}
\end{equation}
 We thus get the following actual fusion rules (see \cite[Conjecture]{Creutzig:2018ogj} for a nice presentation)
 \begin{equation}\nonumber
 \begin{split}
 \sfmod{\ell}{\TypMod{\lambda ; \Delta_{r,s}}} \otimes_V \sfmod{\ell'}{\TypMod{\lambda' ; \Delta_{r',s'}}} &= P\left(\Gr{\sfmod{\ell}{\TypMod{\lambda ; \Delta_{r,s}}}} \fuse \Gr{\sfmod{\ell'}{\TypMod{\lambda' ; \Delta_{r',s'}}}}\right), \qquad \lambda \notin \{ \lambda_{r, s}, \lambda_{u-r, v-s}\} \\
\sfmod{\ell}{\StagMod{r, s}} \otimes_V  \sfmod{\ell'}{\TypMod{\lambda' ; \Delta_{r',s'}}} &= P\left(\Gr{\sfmod{\ell}{\StagMod{r, s}}} \fuse \Gr{\sfmod{\ell'}{\TypMod{\lambda' ; \Delta_{r',s'}}}}\right)\\
\sfmod{\ell}{\StagMod{r, s}} \otimes_V  \sfmod{\ell'}{\StagMod{r', s'}}  &= P\left(\Gr{\sfmod{\ell}{\StagMod{r, s}}} \fuse \Gr{ \sfmod{\ell'}{\StagMod{r', s'}} }\right)
 \end{split}
 \end{equation}
Note that Corollary \ref{cor:typprod} tells us that standard modules form a tensor ideal and so the fusion product $\sfmod{\ell}{\TypMod{{r,s}}^+} \otimes_V \sfmod{\ell'}{\TypMod{{r',s'}}^+}$ can only be a direct sum of projective and modules whose composition factors are of type $\sfmod{\ell''}{\TypMod{{r'',s''}}^+}$. Together with exactness of tensor product it is then also easy to obtain those fusion products. 
The same reasoning also applies to the modules of type $\sfmod{\ell}{\TypMod{{r,s}}^-}$, simply by using the second realization of $L_k(\sltwo)$, where the standard modules are the $\sfmod{\ell}{\TypMod{{r,s}}^-}$ together with the $\sfmod{\ell}{\TypMod{\lambda ; \Delta_{r,s}}}$. In particular $\sfmod{\ell}{\TypMod{{r,s}}^+} \otimes_V \sfmod{\ell'}{\TypMod{{r',s'}}^-}$ is projective. 

We turn to the standard simple modules.   The Grothendieck fusion rules involving the $\sfmod{\ell}{\DiscMod{r,s}^+}$ are (Proposition 18 \cite{CR1})
\begin{align}
\Gr{\sfmod{\ell}{\TypMod{\lambda ; \Delta_{r,s}}}} \fuse \Gr{\sfmod{\ell'}{\DiscMod{r',s'}^+}} &= \sum_{r'',s''} \vfuscoeff{(r,s) (r',s'+1)}{(r'',s'')} \Gr{\sfmod{\ell + \ell'}{\TypMod{\lambda + \lambda_{r',s'} ; \Delta_{r'',s''}}}} \notag \\
&\mspace{55mu} + \sum_{r'',s''} \vfuscoeff{(r,s) (r',s')}{(r'',s'')} \Gr{\sfmod{\ell + \ell' + 1}{\TypMod{\lambda + \lambda_{r',s'+1} ; \Delta_{r'',s''}}}}, \notag\\
\Gr{\sfmod{\ell}{\DiscMod{r,s}^+}} \fuse \Gr{\sfmod{\ell'}{\DiscMod{r',s'}^+}} &= 
\begin{cases}
\displaystyle \sum_{r'',s''} \vfuscoeff{(r,s) (r',s')}{(r'',s'')} \Gr{\sfmod{\ell + \ell' + 1}{\TypMod{\lambda_{r'',s+s'+1}; \Delta_{r'',s''}}}} \\
\displaystyle \mspace{20mu} + \sum_{r''} \vfuscoeff{(r,1) (r',1)}{(r'',1)} \Gr{\sfmod{\ell + \ell'}{\DiscMod{r'',s+s'}^+}}, & \text{if \(s+s'<v\),} \\
\displaystyle \sum_{r'',s''} \vfuscoeff{(r,s+1) (r',s'+1)}{(r'',s'')} \Gr{\sfmod{\ell + \ell' + 1}{\TypMod{\lambda_{r'',s+s'+1}; \Delta_{r'',s''}}}} \\
\displaystyle \mspace{20mu} + \sum_{r''} \vfuscoeff{(r,1) (r',1)}{(r'',1)} \Gr{\sfmod{\ell + \ell' + 1}{\DiscMod{u-r'',s+s'-v+1}^+}}, & \text{if \(s+s' \geqslant v\).}
\end{cases}
\notag
\end{align}
we verify that the $\sfmod{\ell + \ell' + 1}{\TypMod{\lambda_{r'',s+s'+1}; \Delta_{r'',s''}}}$ appearing in the decomposition of $\Gr{\sfmod{\ell}{\DiscMod{r,s}^+}} \fuse \Gr{\sfmod{\ell'}{\DiscMod{r',s'}^+}}$ are all  projective and simple and so 
 one gets
 \begin{equation}\nonumber
 \begin{split}
 &\sfmod{\ell}{\TypMod{\lambda ; \Delta_{r,s}}} \otimes_V \sfmod{\ell'}{\DiscMod{r',s'}^+} = P\left(\Gr{\sfmod{\ell}{\TypMod{\lambda ; \Delta_{r,s}}}} \fuse \Gr{ \sfmod{\ell'}{\DiscMod{r',s'}^+}}\right), \qquad \lambda \notin \{ \lambda_{r, s}, \lambda_{u-r, v-s}\} \\
&\sfmod{\ell}{\DiscMod{r,s}^+} \otimes_V \sfmod{\ell'}{\DiscMod{r',s'}^+} = 
\begin{cases}
\displaystyle \bigoplus_{r'',s''} \vfuscoeff{(r,s) (r',s')}{(r'',s'')} \sfmod{\ell + \ell' + 1}{\TypMod{\lambda_{r'',s+s'+1}; \Delta_{r'',s''}}} \\
\displaystyle \mspace{20mu} + \bigoplus_{r''} \vfuscoeff{(r,1) (r',1)}{(r'',1)} \sfmod{\ell + \ell'}{\DiscMod{r'',s+s'}^+}, & \text{if \(s+s'<v\),} \\
\displaystyle \bigoplus_{r'',s''} \vfuscoeff{(r,s+1) (r',s'+1)}{(r'',s'')} \sfmod{\ell + \ell' + 1}{\TypMod{\lambda_{r'',s+s'+1}; \Delta_{r'',s''}}} \\
\displaystyle \mspace{20mu} + \bigoplus_{r''} \vfuscoeff{(r,1) (r',1)}{(r'',1)} \sfmod{\ell + \ell' + 1}{\DiscMod{u-r'',s+s'-v+1}^+}, & \text{if \(s+s' \geqslant v\).}
\end{cases}
 \end{split}
 \end{equation}

\baselineskip=14pt
\newenvironment{demo}[1]{\vskip-\lastskip\medskip\noindent{\em#1.}\enspace
}{\qed\par\medskip}

\footnotesize{

\flushleft

\bibliographystyle{unsrt}

}
\end{document}